\documentclass{article}

\usepackage{amsmath}
\usepackage{amssymb}
\usepackage{graphicx}
\usepackage{bm}
\usepackage{mathdots}
\usepackage{booktabs}
\usepackage{rotating}
\usepackage{cases}
 \usepackage{bigints}
\usepackage[ruled,linesnumbered]{algorithm2e}
\usepackage[a4paper,left=2.8cm,right=2.8cm,top=2.5cm,bottom=2.5cm]{geometry}
\usepackage{fancyhdr}
\usepackage[framemethod=tikz]{mdframed}
\newtheorem{theorem}{Theorem}[section]
\newtheorem{corollary}{Corollary}[section]
\newtheorem{proposition}{Proposition}[section]
\newtheorem{lemma}{Lemma}[section]
\newtheorem{definition}{Definition}[section]

\newtheorem{remark}{Remark}[section]

\makeatletter % `@' now normal "letter"
\@addtoreset{equation}{section}
\makeatother  % `@' is restored as "non-letter"

\newenvironment{proof}{{\noindent\it Proof.}\quad}{\hfill $\square$\\}
\newcommand{\tabincell}[2]

\usepackage{url}
\usepackage{mathtools}

\begin{document}

\title{Lasso trigonometric polynomial approximation for periodic function recovery in equidistant points}

\author{Congpei An\footnotemark[1]
       \quad \text{and}\quad Mou Cai\footnotemark[2]}

\renewcommand{\thefootnote}{\fnsymbol{footnote}}
\footnotetext[1]{School of Mathematics, Southwestern University of Finance and Economics, Chengdu, China (andbachcp@gmail.com,  ancp@swufe.edu.cn).}
\footnotetext[2]{Wuyang Fengshang, 218 Huilong East Road 2nd Road, Yongchuan District, Chongqing, China (caimou1230@163.com).}

% \date{}
%\shorttitle{SHORT TITLE}
%\shortauthor{F.~FIRSTA AND S.~SECONDA}
% \usepackage{keywords}

\maketitle

\begin{abstract}
This paper introduces a fully discrete soft thresholding trigonometric polynomial approximation on $[-\pi,\pi],$  named Lasso trigonometric interpolation. This approximation is an $\ell_1$-regularized discrete least squares approximation under the same conditions of classical trigonometric interpolation on an equidistant grid. Lasso trigonometric interpolation is a sparse scheme which is efficient in dealing with noisy data. We theoretically analyze Lasso trigonometric interpolation quality for continuous periodic function. The $L_2$ error bound of Lasso trigonometric interpolation is less than that of classical trigonometric interpolation, which improved the robustness of trigonometric interpolation. The performance of Lasso trigonometric interpolation for several testing functions ($\sin$ wave, triangular wave, sawtooth wave, square wave), is illustrated with numerical examples, with or without the presence of data errors.
\end{abstract}



\section{Introduction}\label{sec;introduction}
This paper is aimed at finding a trigonometric polynomial approximation to periodic continuous $f$ on $[-\pi,\pi],$    given noisy data values $f^\epsilon$ at points $x_j \in [-\pi,\pi],\;j=0,\ldots,N,$ using an $\ell_1$-regularized least squares strategy. In practice,  periodic functions are featured in many branches of science and engineering, for example, in digital synthesis, there are four periodic functions of widely using: sin wave,  triangular wave, sawtooth wave,  square wave \cite{Stilson1996AliasFreeDS}.  For their approximation, algebraic polynomials are not appropriate, since algebraic polynomials are not periodic. Naturally, we adopt trigonometric interpolations, which is one of the most special trigonometric polynomial as an approximation strategy. Trigonometric interpolation attract much interest and there is a lot of work can be found in the literature \cite{atkinson2007theoretical,2012Quadrature,Minimum2013,Numerical2008,kress1998numerical,Kress1993,powell1981approximation,trefethen2013approximation,ExpTrapezoidal,ExtPeriodic2015,zygmund2002trigonometric} and references therein. As interpolating nonzero $f\in \mathcal{C}_{2\pi}([-\pi,\pi])$ by $N+1$ equidistant points on $[-\pi,\pi],$ where $\mathcal{C}_{2\pi}([-\pi,\pi])$ is the space of $2\pi$-periodic continuous functions, this trigonometric interpolation of degree $\left[N/2 \right]$ is a numerical discretization of $L_2$ orthogonal projection, and it is highly related to some numerical methods in solving differential and integral equations \cite{pr?ssdorf1991numerical}.

For the case of an equidistant subdivision, we are interested in determining trigonometric interpolation polynomial by the view of least squares approximation problem rather than Lagrange representation \cite{kress1998numerical}. In Lemma  \ref{Lemma1}, we show that the trigonometric interpolation on an equidistant grid is the solution to a discrete least squares approximation problem with trapezoidal rule applied. With elements in equidistant points set and corresponding sampling values of the function deemed as input and output data,  respectively, studies on trigonometric interpolation on an equidistant grid assert that it is an effective approach to modeling mappings from input data to output data. In reality, one can not avoid noisy data in practical sampling. In this paper, we focus on how to recover periodic function with contaminated function values.  Lasso, the acronym for ``least absolute shrinkage and selection operator'', is a shrinkage and selection method for linear regression \cite{Tibshirani2011} which is blessed with the abilities of denoising and feature selection.   Using this technology, we propose $Lasso\;trigonometric\;interpolation\;(LTI),$ which is the analytic solution to an $\ell_1$-regularized least squares problem. In  recently, Lasso hyperinterpolation \cite{Anlasso2021} is discussed to recover continuous function with contaminated data. Both of them are making Lasso incorporating into handle noisy data.  The main difference is that the LTI is only used for dealing with periodic continuous function over the interval $[-\pi,\pi]$ under the condition of classical trigonometric interpolation, but Lasso hyperinterpolation is able to deal with more general continuous function over general regions under the same condition of hyperinterpolation \cite{SLOAN1995238}. As pointed by Hesse and Sloan \cite{2006Hyperinterpolation}, hyperinterpolation is distinct from interpolation, except the case of equal weight quadrature on circle \cite[Section 2]{2006Hyperinterpolation}.  LTI is based on trigonometric interpolation on an equidistant grid, which proceeds all trigonometric interpolation coefficients by a soft threshold operator. Compared with classical trigonometric interpolation on an equidistant grid, LTI is efficient and more stable in recovering noisy periodic function. Moreover, the LTI is more sparse than conventional trigonometric interpolation on an equidistant grid, which makes it  have  advantages in data storage.

We will study approximation properties of LTI and provide error analysis.   It is shown that in the absence of noise, the $L_2$ error of LTI for any nonzero periodic continuous function converges to a nonzero term which depends on the best approximation of this function, whereas such an error of trigonometric interpolation on an equidistant grid  converges to zero as $N\rightarrow \infty.$ However, in the presence of noise, LTI is able to reduce the newly introduced error term caused by noise via multiplying a factor less than one.

This paper is organized as follows. In Section 2, we give the preliminaries on trapezoidal rule and  trigonometric interpolation on an equidistant grid. In Section 3, we constructed the LTI with the help of trapezoidal rule and soft threshold operator.  In Section 4, we analyse the LTI in the view of sparsity. In Section 5, we study the approximation quality of the LTI in terms of the $L_2$ norm.  We give several numerical examples in Section 6 and conclude with some remarks in Section 7.

\section{Preliminaries}
In this section, we review some basic ideas of trigonometric interpolation on an equidistant grid. At first, we state the definition of the degree of trigonometric polynomials, which is different from algebraic polynomials.
\begin{definition}(Definition 8.20 in \cite{Kress1993})\label{def1}
For $d\in \mathbb{N}$ we denote by $\mathbb{T}_d$ the linear space of trigonometric polynomials
$$ q(x)=\sum_{k=0}^d u_k \cos k x+  \sum_{k=1}^d \nu_k \sin k x,$$
with real (or complex) coefficients $u_0,\ldots,u_d$ and $\nu_1,\ldots,\nu_d.$ A trigonometric polynomial $q \in \mathbb{T}_d$ is said to be of degree $d$ if $|u_d|+|\nu_d|>0.$
\end{definition}
It is well known that the cosine functions $c_k(x):=\cos kx, \; k=0,1,\ldots,d,$  and sine functions $s_k(x):=\sin kx, \; k=1,\ldots,d,$ are linearly independent in the continuous periodic function space $\mathcal{C}_{2\pi}([-\pi,\pi]).$ Therefore, the dimension of $\mathbb{T}_d$ defined in Definition \ref{def1} is $2d+1.$

The trapezoidal rule, which is known as the Newton--Cotes quadrature formula of order $1,$ occurs in almost all textbooks of numerical analysis as the most simple and basic quadrature rule  in numerical integration. Although it is inefficient compared with other high order Newton--Cotes quadrature formula and Gauss quadrature formula when integration is applied to a non periodic integrand \cite{kress1998numerical}, it is often exponentially accurate for periodic continuous function \cite{ExtPeriodic2015}. In this paper, trapezoidal rule will play an important role, since the construction of  trigonometric interpolation on an equidistant grid  could be realized by trapezoidal rule.

For any integer $N,$ let $g \in \mathcal{C}_{2\pi}([-\pi,\pi]),$  and $x_j=-\pi+2\pi j/N, \;j=0,1, \ldots, N-1$ be an equidistant subdivision with step $2\pi/N.$  The $N$-point trapezoidal rule takes the form
\begin{align}\label{trapezoidalrule}
\int_{-\pi}^\pi g(x) dx \approx \frac{2\pi}{N}\sum_{j=0}^{N-1}  g(x_j).
\end{align}
Since $g(x_0)=g(x_N)$ (periodicity), no special factors of 1/2 are needed at the endpoints. For concrete example of quadrature points of trapezoidal rule, see Figure 1.
\begin{figure}[htbp]
  \centering
  \includegraphics[width=\textwidth]{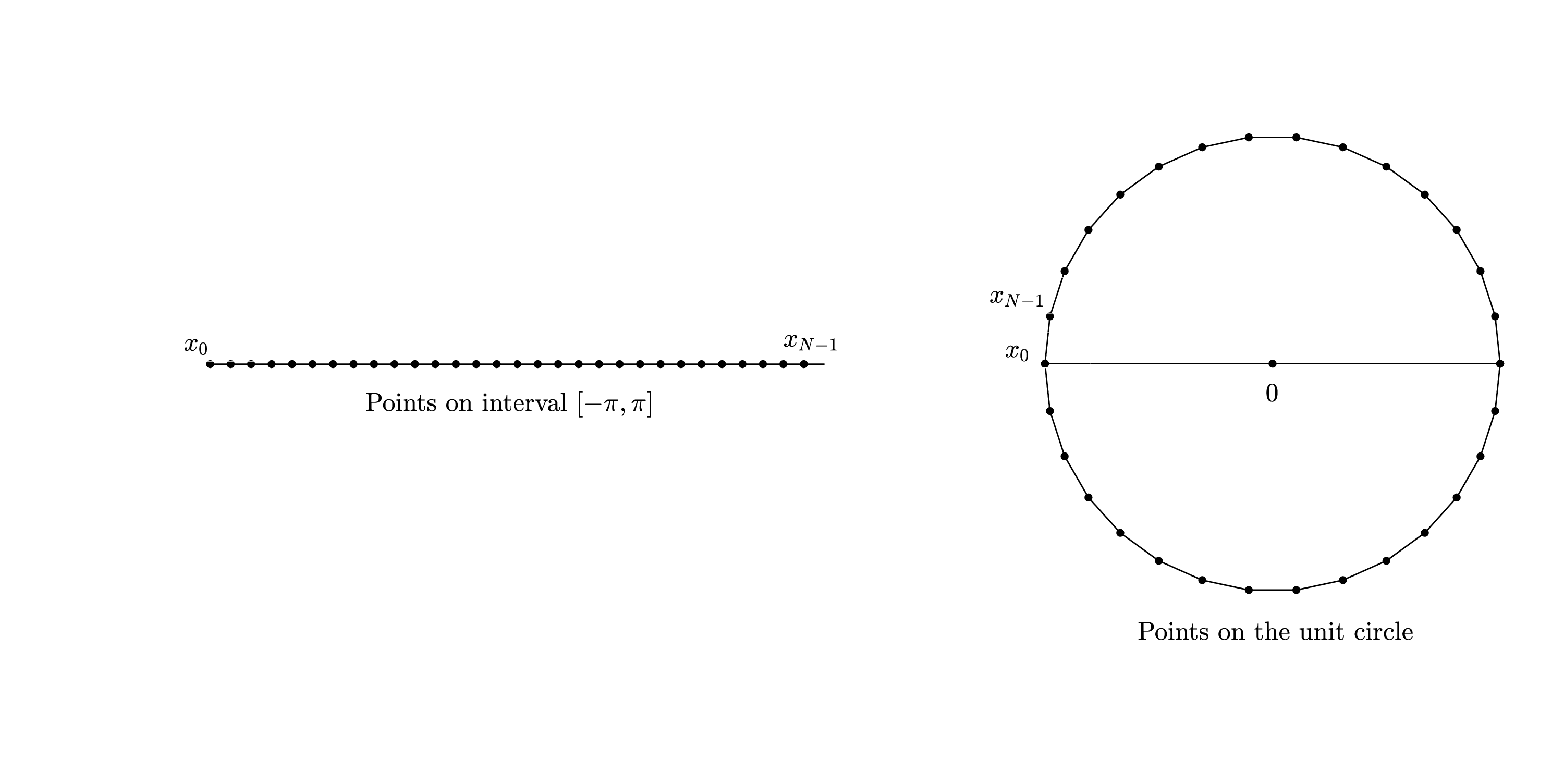}\\
\caption{Fix $N=30,$ the quadrature points of trapezoidal rule on the periodic interval $[-\pi,\pi],$ which can also be regarded as equal diversion points on the unit circle.   }
\end{figure}

\begin{lemma}(Corollary 3.3 in \cite{ExpTrapezoidal})
If $p$ is a trigonometric polynomial of degree $n,$ the $N$-point trapezoidal rule \eqref{trapezoidalrule} is exact for all $n<N,$ i.e.,
\begin{align}\label{exactness}
\int_{-\pi}^\pi p(x) dx =\frac{2\pi}{N} \sum_{j=0}^{N-1} p(x_j) \;\;\;\;\forall p \in \mathbb{T}_n.
\end{align}

\end{lemma}

As a prelude to the introduction of trigonometric interpolation, it is convenient to introduce the $L_2$ orthogonal projection of periodic function, which is theoretically simpler but harder to  compute \cite{2006Hyperinterpolation}.
The best approximation to $f\in \mathcal{C}_{2\pi}[-\pi,\pi]$ in the sense of $L_2$ norm with respect to the space of trigonometric polynomials of degree at most $L$ is given by  \cite{kress1998numerical,ExtPeriodic2015,zygmund2002trigonometric}

 \begin{align}\label{trigprojection}
\mathcal{P}_L f(x):=\sum_{\ell=-L}^{L} c_\ell e^{i\ell x},
\end{align}
with Fourier coefficients
\begin{align}\label{Fouriercofficients2}
c_\ell=\frac{1}{2\pi} \int_{-\pi}^{\pi} f(x) e^{-i\ell x} dx.
\end{align}
The $L_2$ norm is defined by
 \begin{align*}
\Vert f\Vert_{L_2} :=\left<f(x), f(x) \right>_{L_2}^{\frac{1}{2}}=\left( \int_{-\pi}^{\pi}  |f(x)|^2 dx \right)^{\frac{1}{2}},
\end{align*}
where the $L_2$ inner product $\left<f(x),g(x) \right>_{L_2} :=\int_{-\pi}^{\pi} f(x) \overline{g(x)} dx.$ All coefficients in our discussions are in general complex, though in cases of certain symmetries they will be purely real or imaginary.  The trigonometric interpolation of $N+1$ equidistant points of interval is obtained by approximating the coefficients in \eqref{Fouriercofficients2} by  trapezoidal rule \cite{Henrici1979,ExtPeriodic2015}.
\begin{definition}\label{def:triginter }
Given $N+1$ equidistant points $x_0,\ldots,x_N \in [-\pi,\pi],$ and $N+1$ values $f(x_j) \in \mathbb{R}.$  The trigonometric interpolation polynomial on an equidistant grid is given by
\begin{numcases}{\mathcal{T}_n f(x):=}
\frac{1}{\sqrt{2\pi}}\sum_{\ell=-n}^{n} \tilde{c}_\ell e^{i\ell x}, &$N$\; is\;odd, \label{trigonometricinterpolation} \\
\frac{1}{\sqrt{2\pi}} \sum^{n}_{\ell=-n} {}^{'} \tilde{c}_\ell e^{i\ell x}, &$N$\; is\;even, \label{trigonometricinterpolation1}
\end{numcases}
where $n$ is the integer part of $N/2,$ i.e., $n=\left[ N/2\right],$ and it is assumed that $\tilde{c}_n=\tilde{c}_{-n}$ for $N$ is even; the prime `` $\prime $'' indicates that the terms corresponding to $\ell=\pm n$ are to be multiplied by $1/2;$   $\tilde{c}_\ell$ is a discrete version of the coefficients $c_\ell$ in \eqref{Fouriercofficients2} obtained by application of $N$-point trapezoidal rule \eqref{trapezoidalrule}:
\begin{align}\label{cofficientsTrig}
\tilde{c}_\ell=\frac{\sqrt{2\pi}}{N}\sum_{j=0}^{N-1} f(x_j) e^{-i\ell x_j},  \;\; \ell=-n,\ldots,n.
\end{align}
\end{definition}
The next result is concerned with trigonometric polynomial interpolation conditions. Note that the two endpoints of $[-\pi,\pi]$ is regarded as one interpolation point due to the periodicity.
\begin{proposition}\label{intercondition}
Let $\mathbb{T}_n$ be the linear space of trigonometric polynomials of degree at most $n.$ Then the trigonometric polynomial $\mathcal{T}_n f(x)\in \mathbb{T}_n$ defined by Definition 2.2 satisfies the following properties:\\

1. \;\;\;The dimension of trigonometric polynomial space $\mathbb{T}_n$ is equal to the number of interpolation points, i.e., ${\rm dim}\; \mathbb{T}_n=N.$\\

2. \; $\mathcal{T}_n f(x)$ through interpolation points, i.e., $\mathcal{T}_n f(x_j)=f(x_j), \quad \quad j=0,1,\ldots,N-1. $

\end{proposition}
\begin{proof}
The dimension of $\mathbb{T}_n$ can be obtained directly by Definition 2.2. When $N$ is odd, it is clear that
\begin{align*}
\mathcal{T}_n f(x_j)&=\frac{1}{N} \sum_{\ell=-n}^{n} \left( \sum_{k=0}^{N-1} f(x_k) e^{-i\ell x_k}  \right) e^{i\ell x_j} \\
                    &= \frac{1}{N} \sum_{k=0}^{N-1} f (x_k) \left( \sum_{\ell=-n}^n e^{-i\ell x_k} e^{i\ell x_j} \right)= \frac{1}{N} \sum_{k=0}^{N-1} f (x_k) \left( \sum_{\ell=-n}^n e^{i\ell t_k} \right),
\end{align*}
where $t_k=x_j-x_k=\frac{2\pi(j-k)}{N}.$ In parallel, when $N$ is even, we have
\begin{align*}
\mathcal{T}_n f(x_j)&= \frac{1}{N} \sum_{k=0}^{N-1}  f (x_k) \left( \sum_{\ell=-n}^n{}^{'} e^{i\ell t_k} \right),
\end{align*}
for $e^{it_k} \neq 1$ we obtain
\begin{align*}
\sum_{\ell=-n}^n e^{i\ell t_k}=\sum_{\ell=-n}^n{}^{'} e^{i\ell t_k} =e^{-in t_k} \left( \frac{1-e^{it_k N}}{1-e^{it_k}} \right)=0,
\end{align*}
whereas for $e^{i t_k} =1$ each term in the sum is equal to one. Thus,
$$\mathcal{T}_n f(x_j)=  f (x_j).$$
\end{proof}
\begin{remark}
The assumption that $\tilde{c}_{-n}=\tilde{c}_{n}$ in \eqref{trigonometricinterpolation1} is crucial. When $N$ is even,  we have to determine $N+1$ coefficients of trigonometric interpolation, however, we have only $N$ interpolation conditions, which means that  $N$ linear equations have $N+1$ unknown variables. Thus, we need add another condition to reduce the number of variables to ensure the uniqueness of trigonometric interpolation.
\end{remark}

Functions that are even with respect to the center of the interval of interpolation are important in practice. In fact, with the help of symmetry, we have the following definition of trigonometric interpolation about even function.
\begin{definition}
Under conditions of Definition 2.2. The trigonometric interpolation polynomial for even function  on an equidistant grid is given by
\begin{align}\label{def:triginter1}
\mathcal{E}_n f(x):=\frac{1}{\sqrt{\pi}}\sum_{\ell=0}^{n} (\xi_\ell t_\ell) \cos\ell x,
\end{align}
where 
$$t_\ell:=\frac{\sqrt{2\pi}}{N}\sum_{j=0}^{N-1} f(x_j) \cos \ell x_j \;\;\; \text{and} \;\;\;
\xi_\ell:=\left\{
\begin{aligned}
 &1, \;\;\;\;\;\;\;\;\;\;\;\;\ell=0, \;N/2, \\\\
& \sqrt{2},  \;\; \;\;\;\;\;\;\;else.
\end{aligned}
\right.\;\;\;\;.
$$
\end{definition}

It is useful to note that the trigonometric interpolation on an equidistant grid is a projection, i.e.,
$$\mathcal{T}_n ( \mathcal{T}_n f)=\mathcal{T}_n f \in \mathbb{T}_n.$$

\section{Lasso trigonometric interpolation}
\subsection{Least squares approximation}

To introduce Lasso trigonometric interpolation, we show that trigonometric interpolation $\mathcal{T}_nf$ is a solution to a least squares approximation problem. Unless we have special illustration, we set all $n=\left[N/2\right]$ in this paper.  We draw out attention on the following discrete least squares approximation problem:
\begin{equation}\label{Trigapproximation1}
N\; {\rm is \; odd}: \quad\min_{p \in \mathbb{T}_n} \left\{ \frac{1}{2} \sum_{j=0}^{N-1}\frac{2\pi}{N} \left(p(x_j)-f(x_j)\right)^2  \right\}\;\; {\rm with} \;\; p(x)=\frac{1}{\sqrt{2\pi}}\sum_{\ell=-n}^{n} \alpha_{\ell}e^{i \ell x} \in \mathbb{T}_n,
\end{equation}
\begin{equation}\label{Trigapproximation2}
N\; {\rm is \; even}: \quad\min_{q \in \mathbb{T}_n} \left\{ \frac{1}{2} \sum_{j=0}^{N-1}\frac{2\pi}{N} \left(q(x_j)-f(x_j)\right)^2  \right\}\;\; {\rm with} \;\; q(x)=\frac{1}{\sqrt{2\pi}}\sum_{\ell=-n}^{n} {}^{'}\beta_{\ell}e^{i \ell x} \in \mathbb{T}_n,
\end{equation}
where $\beta_{-n}=\beta_n;$  $f$ is a given nonzero continuous $2\pi$-periodic function with values (possibly noisy)  taken at an equidistant points set $\mathcal{X}_{N}=\{x_0,x_1,\ldots,x_{N-1}\}$ on $[-\pi,\pi).$

When $N$ is odd, let $\mathbf{A} \in \mathbb{C}^{N\times N}$ be a matrix with elements $\left[ \mathbf{A} \right]_{j+1,\ell+n+1}=e^{i \ell x_j}/\sqrt{2\pi},\; j=0,1,\ldots,N-1$ and $\ell=-n,\ldots,n,$ and let $\mathbf{W} \in \mathbb{R}^{N\times N} $ be ${\rm diag} \left( \frac{2\pi}{N},\ldots,\frac{2\pi}{N}\right).$ The approximation problem \eqref{Trigapproximation1} can be transformed into the least squares approximation problem
\begin{align}\label{TrigapproxMatrix1}
\min_{\boldsymbol{\alpha} \in \mathbb{C}^N} \frac{1}{2} \Vert \mathbf{W}^{1/2}(\mathbf{A}\boldsymbol{\alpha}-\mathbf{f})\Vert_2^2,
\end{align}
where $\boldsymbol{\alpha}=[\alpha_{-n},\ldots,\alpha_n  ]^T \in \mathbb{C}^N $ is a collection of coefficients $\{\alpha_\ell \}_{\ell=-n}^n$ in constructing $p,$ and $\mathbf{f}=\left[f(x_0),\ldots,f(x_{N-1})  \right]^T \in \mathbb{R}^N$ is a vector of sampling values $\{ f(x_j) \}_{j=0}^{N-1}$ on $\mathcal{X}_N.$

When $N$ is even. Let $\mathbf{B} \in \mathbb{C}^{N\times N}$ be a matrix with elements

 \begin{align*}
 \left[ \mathbf{B} \right]_{j+1,\ell+n}=\frac{e^{i \ell x_j}}{\sqrt{2\pi}} \qquad {\rm and} \qquad \left[ \mathbf{B} \right]_{j+1,N}=\frac{\cos \frac{N}{2} x_j}{\sqrt{2\pi}},
 \end{align*}
 where $j=0,1,\ldots,N-1$ and $\ell=-n+1,\ldots,n-1.$ Then, the approximation problem \eqref{Trigapproximation2} reduces to the least squares approximation problem
\begin{align}\label{TrigapproxMatrix2}
\min_{\boldsymbol{\beta} \in \mathbb{C}^N} \frac{1}{2} \Vert \mathbf{W}^{1/2}(\mathbf{B}\boldsymbol{\beta}-\mathbf{f})\Vert_2^2,
\end{align}
where $\boldsymbol{\beta}=[\beta_{-n},\ldots,\beta_{n-1} ]^T \in \mathbb{C}^N $ is a collection of coefficients $\{\beta_\ell \}_{\ell=-n}^{n-1}$ in constructing $q.$
\begin{lemma}\label{Lemma1}
Given nonzero $f\in \mathcal{C}_{2\pi}([-\pi,\pi])$ and $N$ equidistant points on $[-\pi,\pi),$ and let $\mathcal{T}_n f \in \mathbb{T}_n$ be defined by Definition 2.2.  Then \eqref{trigonometricinterpolation} is the unique solution to  \eqref{TrigapproxMatrix1}, and \eqref{trigonometricinterpolation1} is the unique solution to  \eqref{TrigapproxMatrix2}.
\end{lemma}
\begin{proof}
 Taking the first derivative of the objective function in problem \eqref{TrigapproxMatrix1} and \eqref{TrigapproxMatrix2} with respect to $\boldsymbol{\alpha}, \boldsymbol{\beta},$  respectively, leads to the first order condition
\begin{align}\label{firstoder}
\mathbf{A^*WA} \boldsymbol{\alpha}-\mathbf{A^*Wf}=\mathbf{0} \qquad {\rm and} \qquad \bf{B^*WB} \boldsymbol{\beta}-\bf{B^*Wf}=\bf{0},
\end{align}
where $*$ denotes conjugate transpose.  Note that the assumption $f\neq 0$ implies $\Vert \bf{A^*Wf}\Vert_\infty$ and $ \Vert \bf{B^*Wf}\Vert_\infty$ are nonzero vectors. On the one hand, with the help of exactness \eqref{exactness}, both of the matrix $\bf{A^*WA}$ and $\bf{B^*WB}$ are identity matrix, as all entries of them satisfy
$$
[\mathbf{A^*WA}]_{\ell+n+1, \ell'+n+1}=[\mathbf{B^*WB}]_{\ell+n+1, \ell'+n+1}=\frac{1}{N} \sum_{j=0}^{N-1} e^{i\ell x_j} e^{-i\ell' x_j}=\frac{1}{2\pi}\int_{-\pi}^{\pi} e^{i\ell x} e^{-i\ell' x} dx=\delta_{\ell \ell'},
$$
where $\delta_{\ell\ell'}$ are the Kronecker delta,  $\ell,\ell'=-n,\ldots, n$  and when $N$ is even, $\ell,\; \ell'$ can't equal to $n$ at the same time. The third equality holds from $e^{i\ell x} e^{-i \ell' x} \in \mathbb{T}_{N-1} \subseteq \mathbb{T}_{N}. $ In particular,
$$ \left[\mathbf{B^*WB} \right]_{N,N}=\frac{2\pi}{N} \sum_{j=0}^{N-1}  \left( \frac{\cos \frac{N}{2}x_j}{\sqrt{2\pi}}\right)^2=1.$$
On the other hand, the components of $\bf{A^*Wf}, \;\bf{B^*Wf}$  can be  expressed by
\begin{align*}
&\left[ \bf{A^*Wf}\right]_{\ell+n+1}=\frac{\sqrt{2\pi}}{N} \sum_{j=0}^{N-1}e^{-i\ell x_j} f(x_j), \; \ell=-n,\ldots,n,\\
&\left[ \bf{B^*Wf}\right]_{\ell'+n}=\frac{\sqrt{2\pi}}{N} \sum_{j=0}^{N-1}e^{-i\ell' x_j} f(x_j), \; \ell'=-n+1,\ldots,n-1,
\end{align*}
and
\begin{align*}
\left[ \bf{B^*Wf}\right]_{N}=\frac{\sqrt{2\pi}}{N} \sum_{j=0}^{N-1}\cos (\frac{N}{2} x_j) f(x_j).
\end{align*}
Hence, the trigonometric polynomials constructed with coefficients $\boldsymbol{\alpha}=\bf{A^*Wf}$ or $\boldsymbol{\beta}=\bf{B^*Wf}$ are indeed $\mathcal{T}_n f. $
\end{proof}

Note that the assumption $\beta_{-n}=\beta_{n}$ in \eqref{Trigapproximation2}  ensures that the matrix in the equivalent form of \eqref{Trigapproximation2} is nonsingular, and when $f(x)$ is even function, the assumption will be replaced by $\beta_{-\ell}=\beta_{\ell}, \;\ell=-n,\ldots,n.$
\begin{lemma}\label{Lemma2}
Given even $f\in \mathcal{C}_{2\pi}([-\pi,\pi])$ with respect to the center of interval $[-\pi,\pi].$  The coefficients of trigonometric interpolation $\xi_0 t_0,\ldots,\xi_n t_n$ in  \eqref{def:triginter1} is the unique solution to the following least squares problem
\begin{align}\label{TrigapproxMatrix3}
\min_{\boldsymbol{\eta} \in \mathbb{C}^{n+1}} \frac{1}{2} \Vert \mathbf{W}^{1/2}(\mathbf{C}\boldsymbol{\eta}-\mathbf{f})\Vert_2^2,
\end{align}
where $\left[ \mathbf{C}\right]_{j+1,\ell+1}=\frac{\tau_\ell}{\sqrt{\pi}} \cos \ell x_j, \;\ell=0, \ldots,n;$ $j=0, \ldots,N-1;\; \tau_\ell $ is defined by
$$
\tau_\ell:=\left\{
\begin{aligned}
 &\frac{1}{\sqrt{2}}, \;\;\;\;\;\ell=0,\; N/2, \\\\
& 1,  \;\; \;\;\;\;\;\;\;else,
\end{aligned}
\right.
$$
 and $\mathbf{W},\;\mathbf{f}$ are the same definition in \eqref{TrigapproxMatrix1}.
\end{lemma}
\begin{proof}
Considering the first order condition of \eqref{TrigapproxMatrix3} with respect to $\boldsymbol{\eta},$ we have
\begin{align*}
\left [ \mathbf{C^* W C }\right]_{\ell+1,\ell'+1}=\frac{2\tau_\ell^2}{N} \sum_{j=0}^{N-1} \cos \ell x_j \cos \ell' x_j,
\end{align*}
where $\ell,\ell'=0,\ldots,n.$  For $N$ is even, since $ \cos \frac{N}{2} x_j =(-1)^j,$ then $ \left [ \mathbf{C^* W C }\right]_{n+1,n+1}=1.$  With the help of exactness \eqref{exactness}, we obtain
\begin{align*}
\left [ \mathbf{C^* W C }\right]_{\ell+1,\ell'+1}= \delta_{\ell \ell'}.
\end{align*}
 For the components of $\mathbf{C^* W f},$ we have
$$\left[\mathbf{C^* W f}\right]_{\ell+1}= \frac{2\tau_\ell \sqrt{\pi}}{N} \sum_{j=0}^{N-1}\cos (\ell x_j) f(x_j),\;\; \ell=0,\ldots,n.$$
Hence, the optimal solution to the least squares problem \eqref{TrigapproxMatrix3} can be expressed by the coefficients of trigonometric interpolation defined by \eqref{def:triginter1}.
\end{proof}
\begin{remark}

We point out similarity between trigonometric interpolation on an equidistant grid and hyperinterpolation on a circle \cite{SLOAN1995238}.  The hyperinterpolation of degree $n$ is a numerical discretization of the $L_2$ orthogonal projection constructed by a quadrature rule that exactly integrates all trigonometric polynomials of degree at most $n$\;(algebraic polynomials is at most $2n$), and thus the hyperinterpolation with the trapezoidal rule \eqref{trapezoidalrule} on a circle is just the same as trigonometric interpolation on an equidistant grid. This explains why hyperinterpolation does not arise in the classical literature. However, hyperinterpolation is distinct from trigonometric interpolation for the more general region \cite{2006Hyperinterpolation,SLOAN1995238}. The difference between interpolation and hyperinterpolation is, that in the latter case the number of quadrature points exceeds the degree of freedom.
\end{remark}

\subsection{Lasso regularization}
Now we start to involve Lasso into  $\mathcal{T}_n f.$  From the original idea of Lasso \cite{Lasso1996}, it restricts the sum of absolute values of coefficients to be bounded by some positive number, say, $\eta.$ Then, the trigonometric polynomial which is incorporating Lasso into $\mathcal{T}_n f$ can be achieved via solving the constrained least squares problem
\begin{align}\label{lassoleast}
\min_{p\in \mathbb{T}_n} \left\{ \frac{1}{2} \sum_{j=0}^{N-1} \frac{2\pi}{N} \left( p(x_j)-f(x_j)\right)^2 \right\} \;\;{\rm subject\;to}\; \sum_{\ell=-n}^n |\alpha_\ell| \leq \eta.
\end{align}
It is natural to transform \eqref{lassoleast} into the following $\ell_1$-regularized least squares problem
\begin{align}\label{Lassoproblem}
\min_{p\in \mathbb{T}_n}  \left\{ \frac{1}{2} \sum_{j=0}^{N-1} \frac{2\pi}{N} \left( p(x_j)-f(x_j)\right)^2+\lambda\sum_{\ell=-n}^n |\gamma_\ell| \right\}\quad{\rm with}\quad p=\frac{1}{\sqrt{2\pi}}\sum_{\ell=-n}^n \gamma_\ell e^{i\ell x} \in \mathbb{T}_n,
\end{align}
where $\gamma_{-n}=\gamma_{n}$ for $N$ is even; $\mathcal{X}_{N}=\{x_0,x_1,\ldots,x_{N-1}\}$ is an equidistant points set on $[-\pi,\pi);\;\lambda>0$ is the regularization parameter. And we adopt new notation $\gamma_\ell$ instead of using $\alpha_\ell$ in order to distinguish the coefficients of trigonometric interpolation from those of $\mathcal{T}_n f.$ We also set $\gamma_\ell \in \mathbb{R}$ such that we obtain the first order condition of optimization problem \eqref{Lassoproblem} more easily.
The solution to problem \eqref{Lassoproblem} is our Lasso trigonometric interpolation. Firstly, we give the definition of Lasso trigonometric interpolation. Then we show that the Lasso trigonometric interpolation is indeed the explicit solution to \eqref{Lassoproblem}. Now, we introduce the soft threshold operator \cite{Threshold2011}.

\begin{definition}\label{softthreshold}
The soft threshold operator, denoted by $\mathcal{S}_k(a),$ is
$$\mathcal{S}_k(a):=\max(0,a-k)+\min(0,a+k).$$
\end{definition}
 \noindent We add $\lambda$ as a superscript into $\mathcal{T}_n f,$ denoting that this is a regularized version (Lasso regularized) of $\mathcal{T}_n f$ with regularization parameter $\lambda>0.$
\begin{definition}
Given nonzero $f \in \mathcal{C}_{2\pi}([-\pi,\pi]),$ and $N$ equidistant points on $[-\pi,\pi).$  A Lasso trigonometric interpolation of $f$ is defined as

\begin{numcases}{\mathcal{T}^{\lambda}_n f(x):=}
\frac{1}{\sqrt{2\pi}}\sum_{\ell=-n}^{n} \mathcal{S}_\lambda(\tilde{c}_\ell) e^{i\ell x}, &$N$\; is\;odd, \label{Lassotriginterpolation} \\
\frac{1}{\sqrt{2\pi}} \sum^{n}_{\ell=-n} {}^{'} \mathcal{S}_\lambda(\tilde{c}_\ell) e^{i\ell x}, &$N$\; is\;even, \label{Lassotriginterpolation1}
\end{numcases}
and it is assumed that $\mathcal{S}_\lambda(\tilde{c}_{-n})=\mathcal{S}_\lambda(\tilde{c}_{n})$ for $N$ is even. In particular, when $f$ is an even function with respect to the center of \;$[-\pi,\pi].$  The Lasso trigonometric interpolation of $f$ is defined as
\begin{align}\label{Lassotriginterpolation2}
\mathcal{E}^{\lambda}_n f(x):=\frac{1}{\sqrt{\pi}}\sum_{\ell=0}^{n} \mathcal{S}_\lambda(\xi_\ell t_\ell) \cos \ell x,
\end{align}
where $t_\ell,\;\xi_\ell$ are the same definition in \eqref{def:triginter1},\;$\lambda>0.$
\end{definition}
  The logic of Lasso trigonometric interpolation is to process each coefficient $\tilde{c}_\ell$ or $\xi_\ell t_\ell$ of trigonometric interpolation on an equidistant gird by a soft threshold operator $\mathcal{S}_\lambda(\cdot).$ Similar to the process from \eqref{Trigapproximation1} to \eqref{TrigapproxMatrix1}, problem \eqref{Lassoproblem} can also be transformed into
\begin{align}\label{LassotrigMatrix}
 \min_{\boldsymbol{\gamma} \in \mathbb{R}^{N}} \frac{1}{2} \Vert \mathbf{W}^{1/2}(\mathbf{A}\boldsymbol{\gamma}-\mathbf{f})\Vert_2^2+\lambda\Vert \boldsymbol{\gamma}\Vert_1,\;\; \lambda>0, \;\; N\; {\rm is\;odd}
 \end{align}
 and
 \begin{align}\label{LassotrigMatrix1}
\min_{\boldsymbol{\gamma} \in \mathbb{R}^{N}} \frac{1}{2} \Vert \mathbf{W}^{1/2}(\mathbf{B}\boldsymbol{\gamma}-\mathbf{f})\Vert_2^2+\lambda\Vert \boldsymbol{\gamma}\Vert_1,\;\; \lambda>0, \;\; N\; {\rm is\;even},
\end{align}
where $\boldsymbol{\gamma}=[\gamma_{-n},\ldots,\gamma_{n}   ]^T \in \mathbb{R}^{N},$  and $\boldsymbol{\gamma}=[\gamma_{-n},\ldots,\gamma_{n-1}   ]^T \in \mathbb{R}^{N}$ for $N$ is even. As the convex term $\Vert \cdot\Vert_1$ is nonsmooth, some subdifferential calculus of convex functions \cite{atkinson2007theoretical} is needed. Then we have the following result.
\begin{theorem}\label{theorem:lasso}
Let $\mathcal{T}^\lambda _n f\in \mathbb{T}_n$ be defined by Definition 3.2, and adopt conditions of Lemma 3.1. Then \eqref{Lassotriginterpolation} is the solution to \eqref{LassotrigMatrix}, and \eqref{Lassotriginterpolation1} is the solution to \eqref{LassotrigMatrix1}.
\end{theorem}
\noindent The method of the following proof is based on Theorem 3.4 in \cite{Anlasso2021}.\\
\begin{proof}
 For the even length case, i.e., $N$ is odd, in Lemma 3.1, we have proved that $\bf{A^*WA}$ is an identity matrix, and $\boldsymbol{\alpha}=\bf{A^*Wf}$ is the coefficients of trigonometric interpolation $\mathcal{T}_n f(x)$ defined by \eqref{trigonometricinterpolation}. Then $\boldsymbol{\gamma}=[\gamma_{-n},\ldots,\gamma_{n}       ]^T$ is a solution to \eqref{LassotrigMatrix} if and only if

\begin{align}\label{firstorder1}
\mathbf{0}\in  \mathbf {A^*WA\boldsymbol{\gamma}}-\mathbf{A^*W f}+\lambda \partial(\Vert \boldsymbol{\gamma}\Vert_1)=\boldsymbol{\gamma}-\boldsymbol{\alpha}+\lambda\partial(\Vert \boldsymbol{\gamma}\Vert_1),
\end{align}
where $\partial(\cdot)$ denotes the subdifferential. The first-order condition \eqref{firstorder1} is equivalent to
\begin{align*}
0\in \gamma_\ell-\alpha_\ell+\lambda \partial(|\gamma_\ell|),\;\; \ell=-n,\ldots,n,
\end{align*}
where
\begin{align*}
\partial(|\gamma_\ell|)=\left\{
\begin{aligned}
&1    \qquad\qquad\quad\;\;{\rm if}\; \gamma_\ell>0,    \\
&-1   \qquad\qquad\; {\rm if}\; \gamma_\ell<0,   \\
& \in[-1,1]  \qquad{\rm if}\; \gamma_\ell=0.
\end{aligned}
\right.
\end{align*}
If we denote by $\boldsymbol{\gamma^*}=[\gamma_{-n}^*,\ldots,\gamma_{n}^*]^T $ the optimal solution to \eqref{LassotrigMatrix}, then
\begin{align*}
\gamma_\ell^*=\alpha_\ell-\lambda\partial(|\gamma_\ell^*|), \;\;\;\ell=-n,\ldots,n.
\end{align*}
Thus, there are three cases should be considered:\\

(1) If  $\alpha_\ell>\lambda,$ then $\alpha_\ell-\lambda\partial(|\gamma^*_\ell|)>0;$ thus $\gamma^*_\ell>0,$ yielding $\partial(|\gamma^*_\ell|)=1,$ then
$\gamma^*_\ell=(\alpha_\ell-\lambda)>0.$\\

(2) If $\alpha_\ell<-\lambda,$ then $\alpha_\ell+\lambda\partial(|\gamma^*_\ell|)<0;$ thus $\gamma_\ell^*<0,$ giving $\partial(|\gamma_\ell^*|)=-1,$ then
$\gamma_\ell^*=(\alpha_\ell+\lambda)<0.$\\

(3) If $-\lambda\leq \alpha_\ell\leq \lambda,$ then on the one hand, $\gamma_\ell^*>0$ leads to $\partial(|\gamma_\ell^*|)=1,$ and then $\gamma_\ell^*\leq 0,$  on the other hand,
$\gamma_\ell^*<0$ produces $\partial(|\gamma_\ell^*|)=-1,$ and hence $\gamma_\ell^*\geq 0.$ Two contradictions enforce $\gamma_\ell^* $ to be $0.$
As we hoped, with the aid of soft threshold operator, we obtain
\begin{align*}
\gamma_\ell^*&= \left(\max(0,\alpha_\ell-\lambda)+\min(0,\alpha_\ell+\lambda) \right) \\
&=\mathcal{S}_\lambda( \alpha_\ell),
\end{align*}
which corresponding the coefficients of $\mathcal{T}^\lambda _n f$ defined by \eqref{Lassotriginterpolation}.  For the odd length ($N$ is even), we can get the solution to \eqref{LassotrigMatrix1} is the coefficients of $\mathcal{T}^\lambda _n f$ defined by \eqref{Lassotriginterpolation1} by the same way. In summary, $\mathcal{T}^\lambda _n f$ is indeed the solution to the regularized least squares approximation problem \eqref{LassotrigMatrix} or \eqref{LassotrigMatrix1}.
\end{proof}

\noindent In the case of even function, we have the following result.
\begin{corollary}
When $f$ is an even function with respect to the center of the interval $[-\pi,\pi].$  The coefficients of LTI \eqref{Lassotriginterpolation2} is the analytic solution to the following $\ell_1$-regularized least squares problem
\begin{align}\label{LassotrigMatrix:x}
 \min_{\boldsymbol{\gamma} \in \mathbb{R}^{n+1}} \frac{1}{2} \Vert \mathbf{W}^{1/2}(\mathbf{C}\boldsymbol{\gamma}-\mathbf{f})\Vert_2^2+\lambda\Vert \boldsymbol{\gamma}\Vert_1,\;\; \lambda>0,
 \end{align}
where $\mathbf{W}, \;\mathbf{C},\; \mathbf{f}$ are the same definition in Lemma 3.2.
\end{corollary}
\begin{proof}
From the proof of Lemma \ref{Lemma2}, we know $\mathbf{C^*WC}$ is an identity matrix. Thus, from the proof of Theorem \ref{theorem:lasso}, we have the analytic solution to \eqref{LassotrigMatrix:x} which corresponding to the coefficients of LTI defined by \eqref{Lassotriginterpolation2}.
\end{proof}

\noindent There are three important facts of LTI which distinguishes it from trigonometric interpolation on an equidistant grid.
\begin{remark}
When $\lambda=0,$ coefficients reduce to
\begin{align*}
\mathcal{S}_\lambda (\tilde{c}_\ell)=\tilde{c}_\ell,\;\; \ell=-n,\ldots,n,
\end{align*}
which are coefficients of trigonometric interpolation on an equidistant grid. Thus \eqref{Lassotriginterpolation} and \eqref{Lassotriginterpolation1} could be regarded as a generalization of trigonometric interpolation on an equidistant grid of interval $[-\pi,\pi].$
\end{remark}
\begin{remark}
In practice, the sample data can not be free from noise contaminated. However, it is obviously that the trigonometric interpolation satisfies
$$\mathcal{T}_n f^\epsilon (x_j) = f^\epsilon (x_j),\; 0\leq j\leq N,$$
  where $x_0,x_1,\ldots,x_N$ are equidistant points on $[-\pi,\pi].$ Note that $\mathcal{S}_\lambda(\tilde{c}_\ell) \neq \tilde{c}_\ell$ for $\lambda>0,$  thus the LTI does not satisfy the classical interpolation property, i.e., $\mathcal{T}^\lambda_n f^\epsilon(x_j) \neq f^\epsilon (x_j).$
\end{remark}
\begin{remark}
The fact $\mathcal{S}_\lambda(\tilde{c}_\ell) \neq \tilde{c}_\ell$ for $\lambda>0$ also implies that $\mathcal{T}^\lambda_n p\neq p$ for all $p\in \mathbb{T}_n.$ Hence, $\mathcal{T}^\lambda_n $ is not a projection operator, as $\mathcal{T}^\lambda_n(\mathcal{T}^\lambda_nf) \neq \mathcal{T}^\lambda_nf$ for any nonzero $ f\in \mathcal{C}_{2\pi} ([-\pi,\pi]).$
\end{remark}

\section{Sparsity of the Lasso trigonometric interpolation}
When we approximation continuous periodic problems by the classical trigonometric interpolation on an equidistant grid, we can store it by their coefficients with the form of vector. In real-word problems, these vectors would rarely contain many strict zeros. However, the LTI is relatively sparse than classical trigonometric interpolation on an equidistant grid, since Lasso could produce sparse solutions  by minimizing $\ell_1$ norm. And this fact means the less storage space of the LTI. For topics on sparsity, we refer to \cite{Bruckstein2009,clarke1983optimization}. Now, we are considering the sparsity of LTI of even length \eqref{Lassotriginterpolation}.

 The sparsity of solution $\boldsymbol{\gamma}$ to the $\ell_1$-regularized approximation problem \eqref{LassotrigMatrix} is measured by the number of nonzero elements of $\boldsymbol{\gamma},$ denoted as $\Vert \boldsymbol{\gamma}\Vert_0,$ also known as the zero ``norm'' (it is not a norm actually) of $\boldsymbol{\gamma}.$ If $\lambda>0,$ then $\Vert \mathbf{A^* Wf}\Vert_0$ becomes an upper bound of the number of nonzero elements of $\boldsymbol{\gamma}.$  Moreover, we obtain the exact number of nonzero elements of $\boldsymbol{\gamma}$ with the help of the information of $\boldsymbol{\gamma}.$

\begin{theorem}\label{thm:sparsity}
 Under conditions of Theorem \ref{Lemma1}, let $\boldsymbol{\gamma}=[\gamma_{-n},\ldots,\gamma_n]^T$ be the coefficients of LTI \eqref{Lassotriginterpolation} and $\boldsymbol{\alpha}=\mathbf{A^*Wf}.$  Then the number of nonzero elements of $\boldsymbol{\gamma}$ satisfies

(1) If $\lambda=0,$ then $ \Vert \boldsymbol{\gamma}\Vert_0$ satisfies  $\Vert  \boldsymbol{\gamma}\Vert_0=\Vert \mathbf{A^* Wf}\Vert_0.$ \\

(2) If $\lambda>0,$ then $ \Vert \boldsymbol{\gamma}\Vert_0$ satisfies  $\Vert  \boldsymbol{\gamma}\Vert_0\leq \Vert \mathbf{A^* Wf}\Vert_0,$ more precisely,
\begin{align*}
\Vert \boldsymbol{\gamma}\Vert_0=\Vert \mathbf{A^*Wf}\Vert_0-\#\{\ell:|\alpha_\ell|\leq \lambda \;and \;|\alpha_\ell| \neq 0 \},
\end{align*}
where $ \#\{ \cdot  \}$ denotes the cardinal number of the corresponding set.
\end{theorem}
\begin{proof}
When $\lambda=0,$  the LTI reduces to trigonometric interpolation on the same interpolation points, which suggests their coefficients are totally equal, i.e., $\Vert \boldsymbol{\gamma}\Vert_0=\Vert \mathbf{A^* Wf}\Vert_0.$ When $\lambda>0,$  as the discussion in Theorem \ref{theorem:lasso}, given nonzero entries of $\mathbf{A^*Wf},$ LTI enforces those $\alpha_\ell$ satisfying $|\alpha_\ell|$ to be zero. Hence, assertion (2) holds naturally.
\end{proof}

\noindent Analogously, for the sparsity of LTI of odd length \eqref{Lassotriginterpolation1},  we can obtain the following theorem.
\begin{theorem}
Under conditions of Theorem \ref{Lemma1}, let $\boldsymbol{\gamma}=[\gamma_{-n},\ldots,\gamma_{n-1}]^T$ be the coefficients of LTI \eqref{Lassotriginterpolation1} and $\boldsymbol{\beta}=\mathbf{B^*Wf}.$ Then, the number of nonzero elements of $\boldsymbol{\gamma}$ satisfies

(1) If $\lambda=0,$ then $ \Vert \boldsymbol{\gamma}\Vert_0$ satisfies  $\Vert\boldsymbol{\gamma}\Vert_0=\Vert \mathbf{B^* Wf}\Vert_0.$ \\

(2) If $\lambda>0,$ then $ \Vert \boldsymbol{\gamma}\Vert_0$ satisfies  $\Vert \boldsymbol{\gamma}\Vert_0\leq \Vert \mathbf{B^* Wf}\Vert_0,$ more precisely,
\begin{align*}
\Vert \boldsymbol{\gamma}\Vert_0=\Vert \mathbf{B^*Wf}\Vert_0-\#\{\ell:|\beta_\ell|\leq \lambda \;and \;|\beta_\ell| \neq 0 \},
\end{align*}
where $ \#\{ \cdot  \}$ denotes the cardinal number of the corresponding set.
\end{theorem}

In an extreme case, we could even set large enough $\lambda$ so that all coefficients of LTI \eqref{Lassotriginterpolation} and \eqref{Lassotriginterpolation1} are enforced to be 0. However, as we are constructing a trigonometric polynomial to approximate some $f \in \mathcal{C}_{2\pi}([-\pi,\pi] ),$ this is definitely not the case we desire. Then we have the rather simple but important result, a parameter choice rule for $\lambda$ such that the coefficient is not zero.

\begin{theorem}\label{nonzeroAssumption}
Adopt conditions of Theorem 3.1. If $\lambda< \Vert \bf{ A^*Wf}\Vert_\infty $ (even length)  or $\lambda< \Vert \bf{B^*Wf}\Vert_\infty$ (odd length), then the LTI \eqref{Lassotriginterpolation} or \eqref{Lassotriginterpolation1} is nonzero polynomial.

\end{theorem}
\begin{proof}
When $N$ is odd, suppose to the contrary that the LTI is zero polynomial. Then the first-order condition \eqref{firstorder1} with $\boldsymbol{0}$ gives $\mathbf{A^*Wf} \in \lambda \partial (\| \boldsymbol{0}\|_1 ) =\lambda [ - 1, 1],$ leading to $\| \bf{A^*Wf}\|_ \infty \leq \lambda .$ Hence, its contrapositive also holds: If $\lambda < \| \bf{A^*Wf}\|_ \infty$ , then $\boldsymbol{\gamma}$ could not be $\boldsymbol{0}.$ The case of even $N$ is similar.
\end{proof}

\begin{remark}
For the special case that $f(x)$ is an even function with respect to the center of the interval $[-\pi,\pi],$ obviously, it has totally same consequences as Theorem 4.1 and Theorem 4.2. However, the coefficients of LTI defined by \eqref{Lassotriginterpolation2} are half of the LTI defined by \eqref{Lassotriginterpolation} or \eqref{Lassotriginterpolation1}, which means that the storage space reduce by half. Hence, for even function, we give priority to the LTI defined by \eqref{Lassotriginterpolation2}.
\end{remark}

 Theorem 4.1 and Theorem 4.2 state that the minimization with Lasso is better than minimization without regularization in terms of sparsity, which suggests that the LTI has an advantage over classical trigonometric interpolation on an equidistant grid in data storage.

\section{Error analysis}
We then study the approximation quality of LTI \eqref{Lassotriginterpolation} and  \eqref{Lassotriginterpolation1}  in terms of $L_2$ norm and in the presence of noise. We denote by $ f^\epsilon $ a noisy $f,$ and regard both $f$ and $f^\epsilon$ as continuous for the following analysis. Regarding the noisy version $f^\epsilon$ as continuous is convenient for theoretical analysis, and is always adopted by other scholars in the field of approximation, see, for example, \cite{Parameterchoice2015}. We adopt this trick, and investigate the approximation properties in the sense of $L_2$ error.
The error of best approximation of $f\in\mathcal{C}_{2\pi}([-\pi,\pi])$  by  an element $p$ of $\mathbb{T}_n$ is also involved, which is defined by
\begin{align}
E_n(f):=\inf_{p\in \mathbb{T}_n} \Vert f-p\Vert_\infty.
\end{align}
By the corollary of Fej${\rm{\acute{e}}}$r's theorem (see Corollary 5.4 in \cite{stein2003fourier}), we have the best approximation trigonometric polynomial of degree $n$ to $f,$ i.e., $E_n(f)=\Vert f-p^*\Vert_\infty,$  where $p^*$ denotes this best approximation trigonometric polynomial.

The Lasso trigonometric interpolation $\mathcal{T}^\lambda_n f$ defined by Definition 3.2 and trigonometric interpolation  on an equidistant grid $\mathcal{T}_n f (x)$ defined by Definition 2.2 can be rewritten by following forms:
\begin{align}\label{operator}
\mathcal{T}^\lambda_n f (x):= \frac{1}{\sqrt{2\pi}}\sum_{\ell=-n}^{n+\theta} s_\ell  e^{i\ell x},\qquad   \mathcal{T}_n f (x):= \frac{1}{\sqrt{2\pi}}\sum_{\ell=-n}^{n+\theta} \tilde {c_\ell}  e^{i\ell x},
\end{align}
where

\begin{numcases}{\theta=}
0, & $N$\; is\;odd, \nonumber \\
-1, &$N$\; is\;even, \nonumber
\end{numcases}
and $s_\ell=\mathcal{S}_\lambda(\tilde{c}_\ell),\; \tilde{c_\ell}$ is defined by \eqref{cofficientsTrig}. We first state error bounds for $ \mathcal{T}_{n} f$ for comparison. The following lemma is come from \cite[Problem 8.12]{kress1998numerical}.
\begin{lemma}
Given $ f\in \mathcal{C}_{2\pi} ([-\pi,\pi]),$ and let $\mathcal{T}_{n}f \in \mathbb{T}_{n}$ be defined by \eqref{operator}. Then
\begin{align} \label{hypererrorbound}
\Vert \mathcal{T}_{n} f\Vert_2 \leq \sqrt{2\pi} \Vert f\Vert_\infty,
\end{align}
and
\begin{align}\label{hypererrorbound1}
\Vert \mathcal{T}_{n}f-f\Vert_2 \leq 2\sqrt{2\pi} E_n(f).
\end{align}
Thus $\Vert\mathcal{T}_{n}f-f\Vert_2\rightarrow 0$ as $n\rightarrow \infty.$
\end{lemma}
\begin{proof}
From the definition of the $L_2$ norm, we have
\begin{align}\label{triginerL2error}
\Vert e^{i\ell x}\Vert_2=\left (\int_{-\pi}^\pi e^{i\ell x} e^{-i\ell x} dx\right)^{\frac{1}{2}}=\sqrt{2\pi},
\end{align}
for any nonzero $f \in \mathcal{C}_{2\pi}([-\pi,\pi]),$ we have
\begin{align*}
\Vert \mathcal{T}_n f\Vert_2&= \left \Vert\frac{1}{\sqrt{2\pi}} \sum_{\ell=-n}^{n+\theta} \left(\frac{\sqrt{2\pi}}{N} \sum_{j=0}^{N-1}  f(x_j)e^{-i\ell x_j}  \right)e^{i\ell x} \right \Vert_2 \\
                            &\leq \frac{1}{N} \Vert f\Vert_\infty \Vert e^{i\ell x}\Vert_2 \sum_{\ell=-n}^{n+\theta} \bigg{ \Vert}\sum_{j=0}^{N-1}   1 \cdot e^{-i\ell x_j}  \bigg{\Vert}_2=\Vert f\Vert_\infty \Vert e^{i\ell x}\Vert_2,
\end{align*}
where the equality is due to the exactness \eqref{exactness}. With the help of \eqref{triginerL2error}, we obtain \eqref{hypererrorbound}.

Since $\mathcal{T}_{n} p= p$ for any trigonometric polynomial $p\in \mathbb{T}_n,$  we have
\begin{align}
\Vert \mathcal{T}_{n}f-f\Vert_2 &\leq  \Vert \mathcal{T}_{n}(f-p)+\mathcal{T}_{n}p-f\Vert_2 \nonumber \\
                                &\leq  \sqrt{2\pi}\Vert f-p\Vert_\infty +\Vert p-f\Vert_2, \label{triginerL2error1}
\end{align}
where the second term on the right of inequality \eqref{triginerL2error1} is due to the  Cauchy--Schwarz inequality, which ensures $\Vert g\Vert_2=\sqrt{\left<g,g \right>}\leq \Vert g\Vert_\infty \sqrt{\left<1,1\right>}=\sqrt{2\pi}\Vert g\Vert_\infty$ for all $ g\in\mathcal{C}_{2\pi}([-\pi,\pi]).$ As the inequality holds for arbitrary $p \in \mathbb{T}_n,$  letting $p=p^*$ in \eqref{triginerL2error1} leads to
\begin{align*}
\Vert \mathcal{T}_{n}f-f\Vert_2 \leq  \sqrt{2\pi}E_n (f) +\sqrt{2\pi} E_n (f)=2 \sqrt{2\pi}E_n (f).
\end{align*}
\end{proof}

\noindent Interestingly, one may note the error bounds \eqref{hypererrorbound} and \eqref{hypererrorbound1} are similar to the error bounds of hyperinterpolation on a circle \cite[Example B]{SLOAN1995238}.

\begin{theorem} \label{lemma1}
Adopt conditions of Theorem \ref{nonzeroAssumption}, and let $K(f)=\sum_{\ell=-n}^{n+\theta} \left( \mathcal{S}_\lambda(\tilde{c}_\ell)\tilde{c}_\ell- \mathcal{S}_\lambda^2(\tilde{c}_\ell)   \right).$   Then there exists $\tau_1<1,$ which relies on $f$ and is inversely related to $K(f)$ such that
\begin{align}
\Vert  \mathcal{T}^\lambda_{n}f\Vert_2 \leq \tau_1 \sqrt{2\pi}\Vert f\Vert_\infty,
\end{align}
and there exists $\tau_2<1,$ which relies on $f$ and is inversely related to $K(f-p^*)$  such that
\begin{align}\label{errorbound}
\Vert \mathcal{T}^\lambda_{n} f-f\Vert_2 \leq (1+\tau_2) \sqrt{2\pi} E_n(f)+\Vert p^*-\mathcal{T}^\lambda_{n} p^* \Vert_2.
\end{align}
\end{theorem}
The proof of Theorem \ref{lemma1} is inspired by \cite{Anlasso2021}. \\
\begin{proof}
 With the exactness of $N$-point trapezoidal rule, we have
\begin{align*}
\Vert  \mathcal{T}^\lambda_{n}f\Vert_2^2+  \sum_{j=0}^{N-1} \frac{2\pi}{N} \left|\frac{1}{\sqrt{2\pi}} \sum_{\ell=-n}^{n+\theta}s_\ell e^{i\ell x_j}-f(x_j)\right|^2 &= 2\sum_{\ell=-n}^{n+\theta} (s_\ell^2-s_\ell \tilde{c}_\ell)+\sum_{j=0}^{N-1} \frac{2\pi}{N} f^2(x_j),
\end{align*}
and thus
\begin{align}
\Vert  \mathcal{T}^\lambda_{n}f\Vert_2^2 &\leq  2\sum_{\ell=-n}^{n+\theta} (s_\ell^2- s_\ell \tilde{c}_\ell)+\sum_{j=0}^{N-1} \frac{2\pi}{N} f^2(x_j) \nonumber  \\
                                          &\leq 2\pi \Vert f\Vert_\infty^2 -2\sum_{\ell=-n}^{n+\theta} \left( \mathcal{S}_\lambda(\tilde{c}_\ell)\tilde{c}_\ell- \mathcal{S}_\lambda^2(\tilde{c}_\ell)        \right). \label{Aproximate1}
\end{align}
Note that $|\tilde{c}_\ell|\geq |\mathcal{S}_\lambda(\tilde{c}_\ell)|$ and from the fact that they have the same signs if $\mathcal{S}_\lambda(\tilde{c}_\ell) \neq 0.$ When $\lambda=0,$ we have $\mathcal{S}_\lambda(\tilde{c}_\ell) = \tilde{c}_\ell,$ and when $\lambda$ is so large that $|\tilde{c}_\ell| \leq \lambda,$ we have $\mathcal{S}_\lambda(\tilde{c}_\ell)=0;$ above these make the second term of \eqref{Aproximate1} $>0.$ Then there exists $\tau_1=\tau_1(K(f))<1,$ which is inversely related to $K(f)$ such that
\begin{align*}
\Vert  \mathcal{T}^\lambda_{n}f\Vert_2\leq \left(2\pi\Vert f\Vert^2_\infty -2K(f) \right)^{\frac{1}{2}} = \tau_1 \sqrt{2\pi} \Vert f\Vert_\infty.
\end{align*}
For any trigonometric polynomial $p \in \mathbb{T}_n,$ we have
\begin{align*}
\Vert \mathcal{T}^\lambda_{n} f-f\Vert_2 &=\Vert\mathcal{T}^\lambda_{n} (f-p)-(f-p)-(p-\mathcal{T}^\lambda_{n} p)\Vert_2 \\
&\leq \Vert  \mathcal{T}^\lambda_{n}  (f-p)\Vert_2+\Vert f-p\Vert_2+\Vert p-\mathcal{T}^\lambda_{n} p\Vert_2.
\end{align*}
Since the inequality holds for arbitrary $p \in \mathbb{T}_n,$ and let $p=p^*.$ Then, there exists $\tau_2=\tau_2(K(f-p^*))<1,$ which is inversely related to $K(f-p^*)$ such that
\begin{align*}
\Vert \mathcal{T}^\lambda_{n}  f-f\Vert_2 \leq \tau_2\sqrt{2\pi}\Vert f-p^*\Vert_\infty+\sqrt{2\pi}\Vert f-p^*\Vert_\infty+\Vert p^*-\mathcal{T}^\lambda_{n}  p^*\Vert_2.
\end{align*}
Hence, we obtain the error bound \eqref{errorbound}.
\end{proof}

Comparing with the stability result \eqref{hypererrorbound} and error bound \eqref{hypererrorbound1} of $\mathcal{T}_{n}f,$ it is shown that Lasso can reduce both of them, but an additional regularization error $\Vert p^*-\mathcal{T}^\lambda_{n} p^*\Vert_2$ is introduced in a natural manner. In general, we do not recommend the use Lasso in the absence of noise. As a matter of fact, if the data values are noisy, Lasso will play an important part in reducing noise \cite{Anlasso2021}.

The following theorem describes the denoising ability of $  \mathcal{T}^\lambda_{n}.$
\begin{theorem}
Adopt conditions of Theorem \ref{nonzeroAssumption}. Assume $f^\epsilon \in \mathcal{C}([-\pi,\pi])$ is a noise version of $f,$ and let $\mathcal{T}_n^\lambda f^\epsilon \in \mathbb{T}_n$ be defined by \eqref{operator}.   Then there exists $\tau_3<1,$ which relies on $f$ and $f^\epsilon$ and is inversely related to $K(f^\epsilon-p^*)$ such that
\begin{align}\label{errorboundnoise}
\Vert \mathcal{T}_n^\lambda f^\epsilon-f\Vert_2 \leq \tau_3 \sqrt{2\pi}\Vert f-f^\epsilon\Vert_\infty+(1+\tau_3)\sqrt{2\pi}E_n(f)+\Vert p^*-\mathcal{T}^\lambda_n p^*\Vert_2,
\end{align}
where $p^*$ is the best approximation of $f$ in $\mathbb{T}_n$ over $[-\pi,\pi].$
\end{theorem}
\begin{proof}
For any trigonometric polynomial $p\in \mathbb{T}_n,$ we have
\begin{align*}
\Vert \mathcal{T}_n^\lambda f^\epsilon -f\Vert_2&=\Vert\mathcal{T}_n^\lambda(f^\epsilon-p)-(f-p)-(p-\mathcal{T}_n^\lambda p)\Vert_2 \\
                                               &\leq \Vert \mathcal{T}_n^\lambda(f^\epsilon-p)\Vert_2+\Vert f-p\Vert_2+\Vert p-\mathcal{T}_n^\lambda p\Vert_2.
\end{align*}
Then by Theorem \ref{lemma1}, letting $p=p^*$ gives
\begin{align*}
\Vert \mathcal{T}_n^\lambda f^\epsilon -f\Vert_2 \leq \tau_3 \sqrt{2\pi} \Vert f^\epsilon-p^*\Vert_\infty+\sqrt{2\pi}\Vert f-p^*\Vert_\infty+\Vert p^*-\mathcal{T}_n^\lambda p^*\Vert_2,
\end{align*}
where $\tau_3<1$ is inversely related to $K(f^\epsilon-p^*).$ Estimating $\Vert f^\epsilon-p^*\Vert_\infty $ by $ \Vert f^\epsilon-p^*\Vert_\infty \leq \Vert f^\epsilon-f\Vert_\infty+\Vert f-p^*\Vert_\infty$ gives \eqref{errorboundnoise}.
\end{proof}
\begin{remark}
When there exists noise and $\lambda=0,$ there holds
\begin{align*}
\Vert\mathcal{T}_n f^\epsilon-f\Vert_2 \leq \sqrt {2\pi}\Vert f-f^\epsilon\Vert_\infty+2\sqrt {2\pi}E_n(f),
\end{align*}
which enlarges the part $\tau_3 \sqrt{2\pi} \Vert f-f^\epsilon\Vert_\infty+(1+\tau_3)\sqrt{2\pi}E_n(f)$ in \eqref{errorboundnoise} but vanishes the part $\Vert p^*-\mathcal{T}^\lambda_n p^*\Vert_2.$ Hence there should be a trade-off strategy for $\lambda$ in practice.
\end{remark}

\section{Numerical experiments}
In this section, we report numerical results to illustrate the theoretical results derived from above and test the approximation efficiency of  Lasso trigonometric interpolation $\mathcal{T}^\lambda_n f(x)$ defined by Definition 3.2 and trigonometric interpolation on an equidistant grid $\mathcal{T}_n f(x)$ defined by Definition 2.2. For convenience in programming, the choices of the basis of $\mathbb{T}_n$ are normalized trigonometric system
$$\left\{\frac{1}{\sqrt{2\pi}},\frac{\cos x}{\sqrt{\pi}},\frac{\sin x}{\sqrt{\pi}},\frac{\cos 2x}{\sqrt{\pi}},\frac{\sin 2x}{\sqrt{\pi}},\ldots  \right\} \qquad {\rm and} \qquad    \left\{\frac{1}{\sqrt{2\pi}},\frac{\cos x}{\sqrt{\pi}},\frac{\cos 2x}{\sqrt{\pi}},\ldots  \right\}.$$
We only show the even length case ($N$ is odd), the odd length case ($N$ is even) is similar. Six testing functions are involved in the following experiments: \\

Basic wave functions on $[-\pi,\pi]$ \cite{2006Music}:
\begin{align*}
& Sin\; wave: f_1(x)=\sin(x), \\ \\
& Triangular\; wave: f_2(x)=\left\{
\begin{aligned}
&\frac{3\pi}{2}-x   \qquad  -\pi\leq x\leq 0,    \\
& x-\frac{5\pi}{2} \;\qquad 0\leq x\leq \pi,
\end{aligned}
\right. \\ \\
& Sawtooth\; wave: f_3(x)=\left\{
\begin{aligned}
&(\pi-x)/2   \quad \;\;\qquad 0<x<\pi,    \\
&-(\pi+x)/2   \qquad -\pi<x<0,   \\
& 0  \qquad\qquad\qquad\;\;\;\; x=\pm \pi,
\end{aligned}
\right. \\ \\
& Square\; wave: f_4(x)=\left\{
\begin{aligned}
&-1  \qquad  \; \;-\pi\leq x<0,    \\
&1 \qquad \qquad 0< x\leq \pi, \\
&0 \qquad \qquad  x=0,
\end{aligned}
\right.
\end{align*}
and these functions can be easily generated by MATLAB command \texttt{sin(x)}, \texttt{sawtooth(x,1/2)}, \texttt{sawtooth(x)}, \texttt{square(x)}, with symbolic variable \texttt{x}. \\

Periodic function with high oscillation:
\begin{align}
&f_5(x)=\cos(50x+4\sin(5x)),  \label{highoscillation2}   \\
&\nonumber \\
&f_6(x):  r (u''+u')-\cos(t)u=0.1,\;\;     0\leq t\leq 2\pi,  \label{highoscillation1}  
\end{align}
where $f_5(x)$ is $''FM\; signal''$ \cite{ExtPeriodic2015}, which is the way of carrying audio signals on a radio frequency carrier \cite[Chapter 8]{2006Music}.  $f_6(x)$ is the  numerical solution to periodic ODE \eqref{highoscillation1},  which could be computed by an automatic Fourier spectral method \cite{ExtPeriodic2015},  and parameter $r$ controls the oscillatory degree of $f_6(x)$, the larger $r$ is, the less oscillation of $f_6(x).$  In our next experiments, we set $r=0.001.$ 

We adopt the relative error to test the efficiency of approximation, which is estimated as follows:
\begin{align}\label{relate_error}
\frac{ \Vert f(x)-\mathcal{T}^{\lambda}_n f(x)\Vert_{L_2}}{\Vert f(x)\Vert_{L_2}}\;\;\; {\rm and}\;\;\; \frac{ \Vert f(x)-\mathcal{T}_n f(x)\Vert_{L_2}}{\Vert f(x)\Vert_{L_2}}
\end{align}
with
 \begin{eqnarray*}
  \|f(x)-\mathcal{T}^\lambda_n f(x)\|_{L_2} = &\left(\bigintsss_{-\pi}^{\pi} (f(x)-\mathcal{T}^\lambda_n f(x))^2{\rm d}x\right)^{1/2}\\
                                       \simeq &\left(\frac{2\pi}{N}  \sum_{j=1}^N (f(x_j)-\mathcal{T}^\lambda_n f(x_j))^2\right)^{1/2}.
\end{eqnarray*}
\indent We first test the sparsity of LTI and classical trigonometric interpolation with data sampled on an equidistant points. The level of noise is measured by $signal$-$to$-$noise$\; $radio \;(SNR),$  which is defined as the ratio of signal to the noisy data, and is often expressed in decibels (dB). For given clean signal $\mathbf{f}\in \mathbb{R}^{N\times 1},$ we add noise to this data
\begin{align*}
\bf{d}=\bf{f}+\alpha \bf{\epsilon},
\end{align*}                                                 
where $\alpha$ is a scalar used to yield a predefined SNR, $\bf{\epsilon}$ is a vector following Gaussian distribution with mean value 0. Then we give the definition of SNR:
\begin{align*}
{\rm SNR}=10 \log_{10} \left( \frac{P_{{\rm signal}}}{\alpha P_{{\rm noise}}}      \right),
\end{align*}
where $P_{{\rm signal}}=\sqrt{\frac{1}{N}\sum_{k=1}^N f^2(x_k)},\; P_{\rm{noise}}$ is the standard deviation of $\bf{\epsilon}.$  A lower scale of  SNR suggests more noisy data.  For regularization parameter $\lambda,$  we choose from sequence $10^{-4}:10^{-4}:0.3,$ which makes \eqref{relate_error} minimum. For more advanced and adaptive methods to choose the parameter $\lambda$, we refer to \cite{GenH1979,Parameterchoice2015}. The sparsity is tested by the zero norm $\Vert \boldsymbol{\alpha}\Vert_0,$ the smaller $\Vert \boldsymbol{\alpha}\Vert_0$ means sparser.     Fix $N=501,$ then the degree $n=250.$ Adding 90 dB, 30 dB, and 10 dB Gauss white noise onto sampled data, respectively. From Table 1 and Table 2, we can see $\mathcal{T}_n^\lambda f(x)$ is sparse in most cases. However, $\mathcal{T}_n f(x)$ is non-sparse unless the sample date is free from noise for some functions (This case is impossible in practice). Even though the sampled data is disturbed only by little noise, $\mathcal{T}_n f(x)$ almost have not any sparsity, which means that LTI has absolutely advantage than classical trigonometric interpolation in storing data.

We then test the denoise ability of $\mathcal{T}_n^\lambda f(x)$ and $\mathcal{T}_n f(x).$ Fix $L=(N-1)/2,$ and let $N$ increase from $5$ to $501$ with step $6$ (ensure $N$ is odd).  Computational results plotted in Figure 2 asset trigonometric interpolation \eqref{trigonometricinterpolation} could not be used to approximation periodic functions with seriously noise. Conversely,  LTI can reduce the $L_2$ error of all testing fucntions as the number of interpolation points increasing, especially for the $\sin$ wave and FM signal.

We also test the efficiency of LTI in recovering basic wave functions with 10 dB Gauss white noise.  In the rest numerical experiments, for the number of sample data, we always set $N=501$ and $L=250.$  In Figure 3, the  performance of recovering images and error curves of $\sin$ and triangular wave is amazing, which means that Lasso improves the robustness of trigonometric interpolation with respect to noise in general. However, the LTI couldn't recover sawtooth, square wave as well as above two waves and meanwhile the denoising pictures show there exists Gibbs phenomenon \cite{trefethen2013approximation}.

To solve this difficult, we add Lanczos $\sigma$ factor \cite{lanczos1988applied} into the coefficients of LTI, i.e.,
\begin{align}\label{LanczocLTI}
 \mathcal{T}^{\lambda \sigma}_n f(x) =\frac{1}{\sqrt{2\pi}} \sum_{\ell=-n}^n (\sigma)^p \mathcal{S}_\lambda(\tilde{c}_\ell) e^{i\ell x},                         
 \end{align}
with
$$\sigma=
\left\{\begin{aligned}
& \frac{\sin(\pi \cdot (\ell/L))}{ \pi \cdot (\ell/L)} \;\;\; \ell\neq 0,  \\ \\
& 1                                 \;\;\;\;\qquad \qquad\quad  \ell=0,
\end{aligned}
\right.   
$$
where $p>0.$   For periodic high oscillatory function $f_6(x),$ let $p=5$ in \eqref{LanczocLTI}. In Figure 4, we can see the recovery efficiency of trigonometric interpolation, LTI, and LTI with Lanczoc $\sigma$ factor.  Even though the efficiencies of recovering high oscillatory function of LTI and LTI with Lanczoc $\sigma$ factor are not remarkable compared with the efficiencies of recovering $\sin$ or triangular wave by LTI, it could still reduce noise to some extent. Moreover, the LTI with Lanczoc $\sigma$ factor could significant reduce the oscillation, which is arisen by Gibbs phenomenon. Thus it's worthwhile  combining LTI and  Lanczoc $\sigma$ factor  for recovering periodic high oscillatory function.
\begin{table*}
\centering
\caption{The sparsity of $\mathcal{T}_n f(x)$ for $N=501$}
\begin{tabular}{lllll} % 控制表格的格式
\toprule
   &  \scriptsize{$\Vert \boldsymbol{\alpha} \Vert_0$ of $\mathcal{T}_n f(x)$} &   \scriptsize{$\Vert \boldsymbol{\alpha} \Vert_0$ of $\mathcal{T}_n f(x)$} &   \scriptsize{$\Vert \boldsymbol{\alpha} \Vert_0$ of $\mathcal{T}_n f(x)$}   &   \scriptsize{$\Vert \boldsymbol{\alpha} \Vert_0$ of $\mathcal{T}_n f(x)$}   \\ 
   & \scriptsize{without noise} &  \scriptsize{ with 90 dB noise}  &  \scriptsize{ with 30 dB noise}&  \scriptsize{ with 10 dB noise} \\ 
  \midrule
${\rm Sin}$ wave          & \qquad   1 &\qquad 499 &\qquad 500 &\qquad 500 \\
Triangular wave          & \qquad  251&\qquad 497  &\qquad 501 &\qquad 501  \\
Sawtooth wave             &\qquad  501&\qquad 501 &\qquad 501 &\qquad 501 \\
Square wave                  &\qquad247  &\qquad 499  &\qquad 500 &\qquad 500  \\
  High oscillatory $f_5(x)$  &  \qquad 26  &\qquad 498 &\qquad 501 &\qquad 501  \\
   High oscillatory $f_6(x)$  & \qquad 97 &\qquad 501 &\qquad 501 &\qquad 501   \\
  \bottomrule
  \end{tabular}
  \label{tbl:table1}
\end{table*}

\begin{table*}
\centering
\caption{The sparsity of $\mathcal{T}_n^\lambda f(x)$ for $N=501$}
\begin{tabular}{lllll} % 控制表格的格式
\toprule
   &  \scriptsize{$\Vert \boldsymbol{\alpha} \Vert_0$ of $\mathcal{T}_n^\lambda f(x)$} &   \scriptsize{$\Vert \boldsymbol{\alpha} \Vert_0$ of $\mathcal{T}^\lambda_n f(x)$} &   \scriptsize{$\Vert \boldsymbol{\alpha} \Vert_0$ of $\mathcal{T}^\lambda_n f(x)$}   &   \scriptsize{$\Vert \boldsymbol{\alpha} \Vert_0$ of $\mathcal{T}^\lambda_n f(x)$}   \\ 
   & \scriptsize{without noise} &  \scriptsize{ with 90 dB noise}  &  \scriptsize{ with 30 dB noise}&  \scriptsize{ with 10 dB noise} \\ 
  \midrule
${\rm Sin}$ wave          & \qquad   1 &\qquad 1 &\qquad 7  &\qquad 10 \\
Triangular wave          & \qquad  3&\qquad 89  &\qquad 39  &\qquad 8\\
Sawtooth wave             &\qquad 57&\qquad 440 &\qquad 324 &\qquad 145 \\
Square wave                  &\qquad8 &\qquad 126  &\qquad 170 &\qquad 82 \\
  High oscillatory $f_5(x)$  &  \qquad 13 &\qquad 22 &\qquad 68 &\qquad 65 \\
   High oscillatory $f_6(x)$  & \qquad 63 &\qquad 83  &\qquad 216 &\qquad 184 \\
  \bottomrule
  \end{tabular}
  \label{tbl:table1}
\end{table*}
\begin{figure}[htbp]
  \centering
  \includegraphics[width=\textwidth]{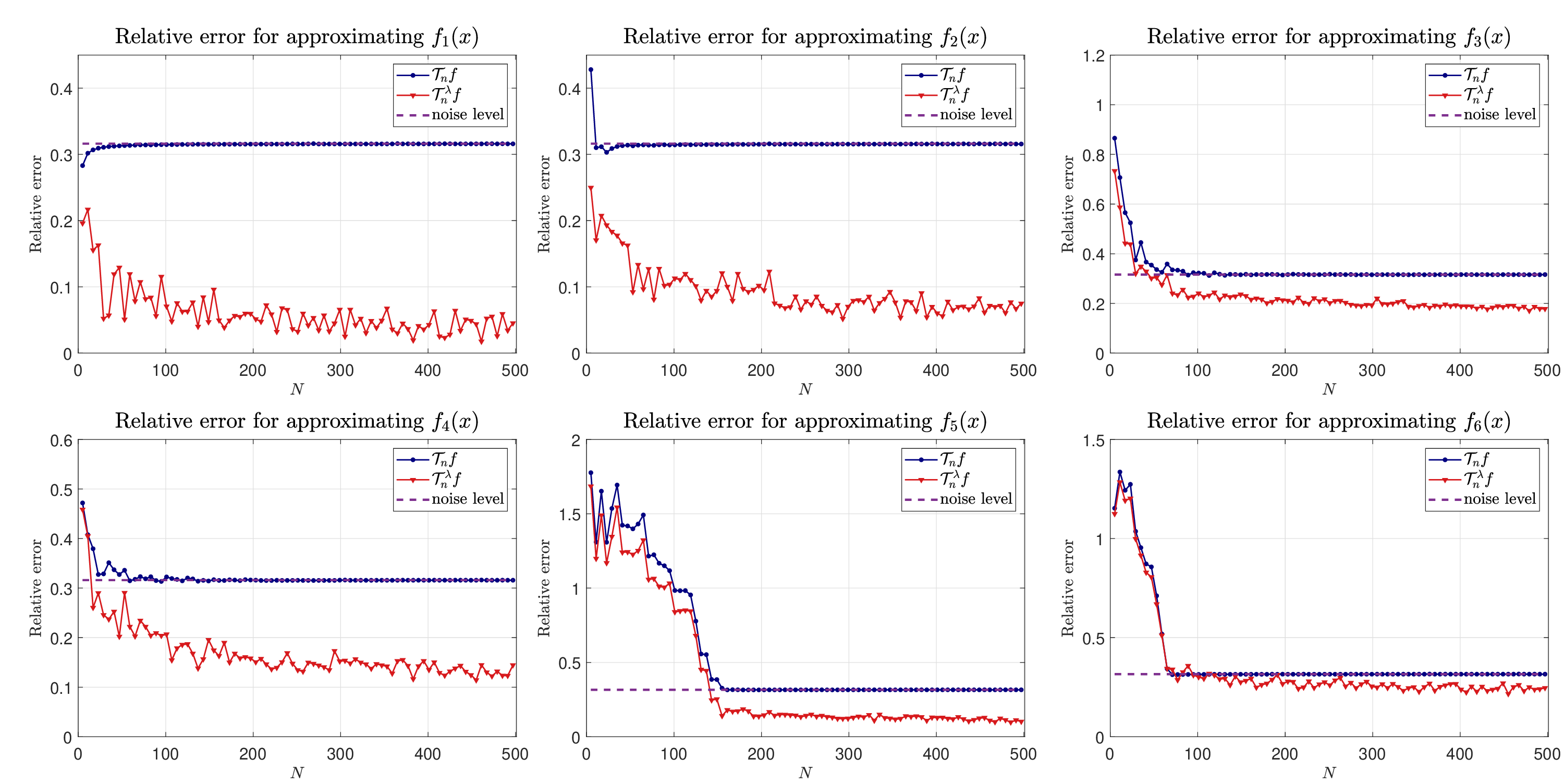}\\
\caption{Computational results for six testing functions on approximation scheme \eqref{trigonometricinterpolation} and \eqref{Lassotriginterpolation2} with increasing $N$  from $5$ to $501$ with step $6$ and fixed $L=(N-1)/2$ in the presence of 10 dB Gauss white noise. As $N$ increased, all curves of relative error show the well performances of approximation scheme \eqref{Lassotriginterpolation2}  in reducing noise.}
\end{figure}

\begin{figure}[htbp]
  \centering
  \includegraphics[width=\textwidth]{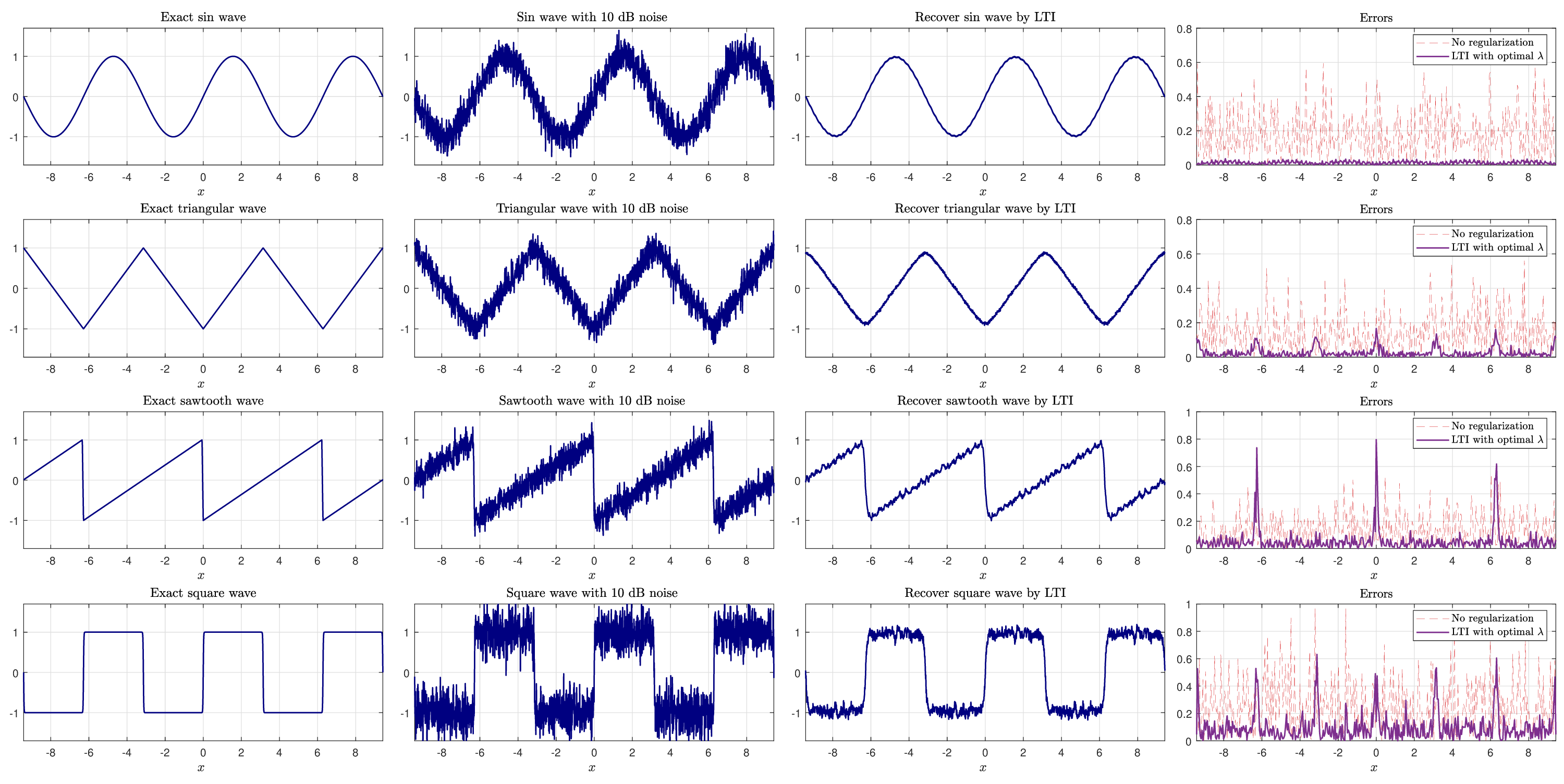}\\
\caption{Approximation results of $\sin$ wave, triangular wave, sawtooth wave, square wave with 10 dB Gauss white noise via LTI.}
\end{figure}

\begin{figure}[htbp]
  \centering
  \includegraphics[width=\textwidth]{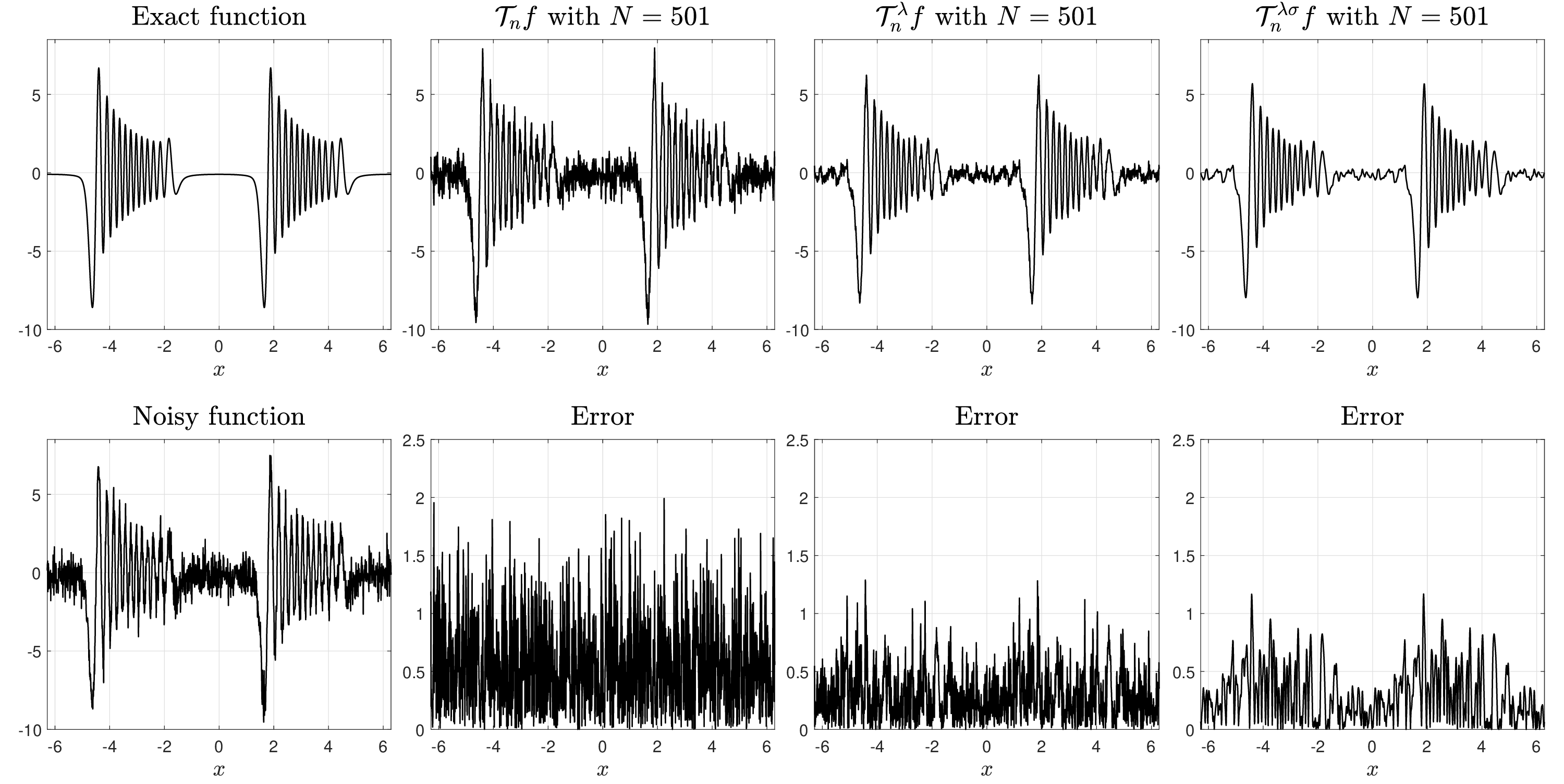}\\
\caption{Approximation results of oscillatory function $f_6(x)$ with 10 dB Gauss white noise via trigonometric interpolation, LTI, LTI with Lanczos $\sigma$ factor.}
\end{figure}

\begin{figure}[htbp]
  \centering
  \includegraphics[width=\textwidth]{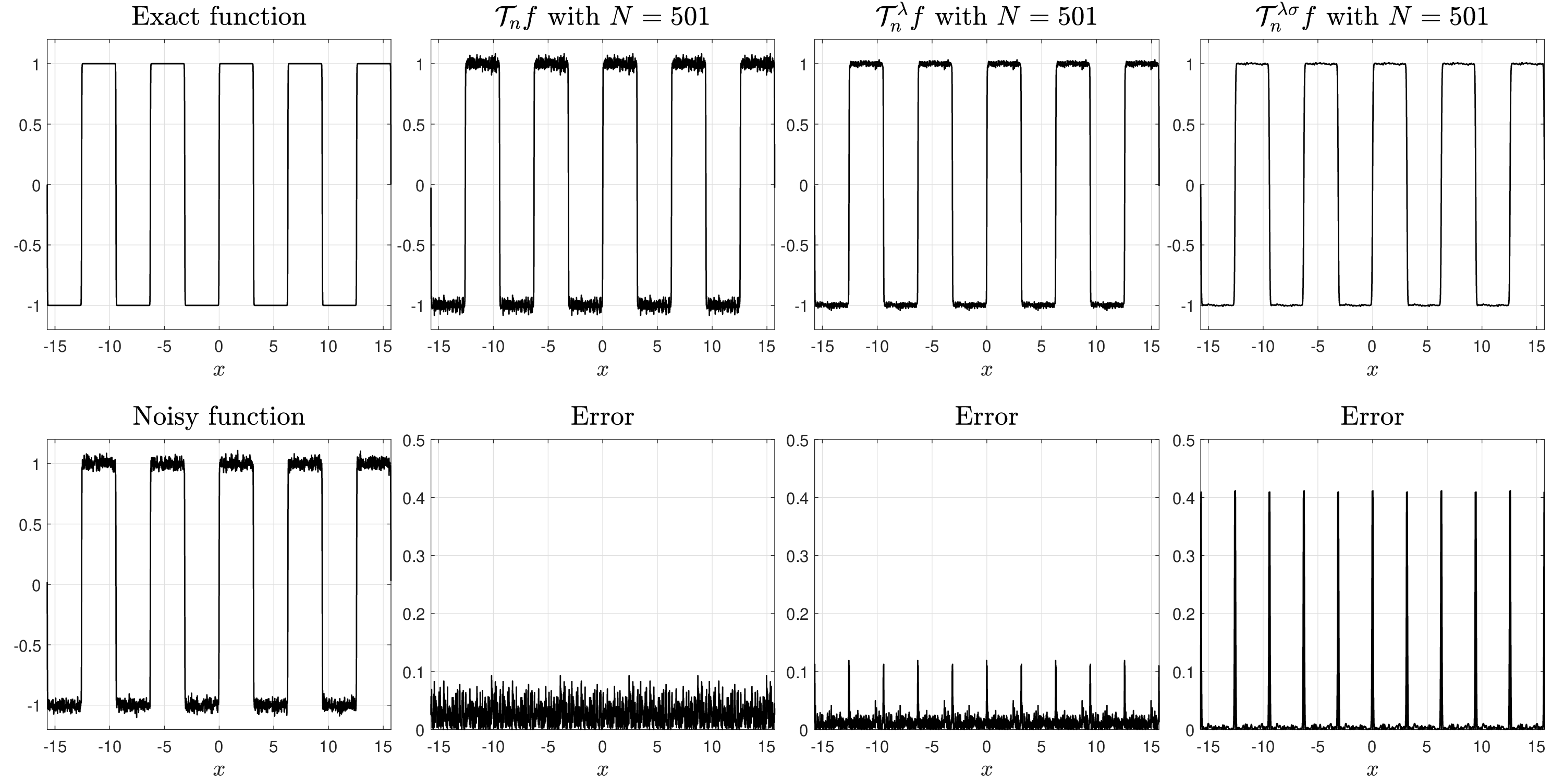}\\
\caption{Approximation results of square wave with 30 dB Gauss white noise via trigonometric interpolation, LTI, LTI with Lanczos $\sigma$ factor.}
\end{figure}

At last, we try to recover square wave and sawtooth wave with lower noise. In fact, recovering these two functions is a great challenge, we couldn't recover them better than images in Figure 3 by any methods in this paper when noise level is high. For lower level noise, say 30 dB Gauss white noise, we just display the efficiency of recovery about square wave (The image of sawtooth wave is similar). Fix $p=30,$  Figure 5 shows that the trigonometric interpolation and LTI could not reduce noise significantly, but the LTI with Lanczoc $\sigma$ factor could recover square wave well in general. However, note that in some certain interval, both LTI and LTI with Lanczoc $\sigma$ factor have great discrepancy with exact function, how to improve or overcoming this phenomenon is still a problem.

\section{Concluding remarks}
What we have seen from the above is that Lasso combined with trigonometric interpolation can reduce noise when the approximation function is periodical.  In particular, we obtain the Lasso trigonometric interpolation, which induces the conventional trigonometric interpolation with the same interpolation points. This new trigonometric interpolation can not only recover periodic continuous function better, but also has characteristic of coefficient sparse compared with normal trigonometric interpolation, which needs more attention in the future.

Although we only consider the $\ell_1$-regularized term, it also provides some useful informations that regularization may improve performance of trigonometric polynomial approximation. In inverse problems, statistics, and machine learning, different kinds of regularization terms are developed \cite{An_2020}. We may consider other regularization techniques and derive other regularized trigonometric interpolation. Last but far from the least, the regularization parameter $\lambda$ deserve future studies  \cite{An_2020}. How to choose a proper regularization parameter $\lambda$? This is an interesting topic in regularization problems, for further study, we may consider the method in \cite{GenH1979,Hansen2007,lazarov2007balancing,lu2010discrepancy,Parameterchoice2015}.

\section*{Acknowledgment}
The authors acknowledge the support of the Tianfu Emei plan (No. 1914) and NSF of China (No.12371099). The second author is greatly thankful to Dr. Andre Milzarek for valuable comments that helped to improve this paper. We are very grateful to the anonymous referees for their careful reading of our manuscript and their many insightful comments.

\bibliographystyle{acm}

\begin{thebibliography}{10}

\bibitem{An_2020}
{\sc An, C., and Wu, H.-N.}
\newblock Tikhonov regularization for polynomial approximation problems in
  {G}auss quadrature points.
\newblock {\em Inverse Problems 37}, 1 (2020), 015008.

\bibitem{Anlasso2021}
{\sc An, C., and Wu, H.-N.}
\newblock {L}asso hyperinterpolation over general regions.
\newblock {\em SIAM Journal on Scientific Computing 43}, 6 (2021),
  A3967--A3991.

\bibitem{atkinson2007theoretical}
{\sc Atkinson, K., and Han, W.}
\newblock {\em Theoretical Numerical Analysis: A Functional Analysis
  Framework}.
\newblock Texts in Applied Mathematics. Springer New York, 2007.

\bibitem{2012Quadrature}
{\sc Beckmann, J., Mhaskar, H.~N., and Prestin, J.}
\newblock Quadrature formulas for integration of multivariate trigonometric
  polynomials on spherical triangles.
\newblock {\em Gem International Journal on Geomathematics 3}, 1 (2012),
  119--138.

\bibitem{2006Music}
{\sc Benson, D.}
\newblock {\em Music: A Mathematical Offering}.
\newblock Cambridge University Press, 2006.

\bibitem{Bruckstein2009}
{\sc Bruckstein, A.~M., Donoho, D.~L., and Elad, M.}
\newblock From sparse solutions of systems of equations to sparse modeling of
  signals and images.
\newblock {\em SIAM Review 51}, 1 (2009), 34--81.

\bibitem{Minimum2013}
{\sc Chandrasekaran, S., Jayaraman, K.~R., and Mhaskar, H.~N.}
\newblock Minimum {S}obolev norm interpolation with trigonometric polynomials
  on the torus.
\newblock {\em Journal of Computational Physics 249\/} (2013), 96--112.

\bibitem{clarke1983optimization}
{\sc Clarke, F.}
\newblock {\em Optimization and Nonsmooth Analysis}.
\newblock Canadian Mathematical Society series of monographs and advanced
  texts. Wiley, 1983.

\bibitem{Numerical2008}
{\sc Dahlquist, G., and Bj{\"o}rck, {\AA}.}
\newblock {\em Numerical Methods in Scientific Computing: Volume 1}.
\newblock Society for Industrial and Applied Mathematics, USA, 2008.

\bibitem{Threshold2011}
{\sc Foucart, S.}
\newblock Hard thresholding pursuit: An algorithm for compressive sensing.
\newblock {\em SIAM Journal on Numerical Analysis 49}, 6 (2011), 2543--2563.

\bibitem{GenH1979}
{\sc Golub, G.~H., Heath, M., and Wahba, G.}
\newblock Generalized {C}ross-{V}alidation as a method for choosing a good
  ridge parameter.
\newblock {\em Technometrics 21}, 2 (1979), 215--223.

\bibitem{Hansen2007}
{\sc Hansen, P.~C.}
\newblock Regularization tools version 4.0 for matlab 7.3.
\newblock {\em Numerical Algorithms 46}, 2 (2007), 189--194.

\bibitem{Henrici1979}
{\sc Henrici, P.}
\newblock Barycentric formulas for interpolating trigonometric polynomials and
  their conjugates.
\newblock {\em Numerische Mathematik 33}, 2 (1979), 225--234.

\bibitem{2006Hyperinterpolation}
{\sc Hesse, K., and Sloan, I.~H.}
\newblock Hyperinterpolation on the sphere. {I}n: {F}rontiers in
  {I}nterpolation and {A}pproximation({D}edicated to the {M}emory of
  {A}mbikeshwar {S}harma) (eds.: N. {K}. {G}ovil, {H}. {N}. {M}haskar, {R}am
  {N}. {M}ohapatra, {Z}uhair {N}ashed and {J}. {S}zabados).
\newblock {\em Champman $\&$ Hall/CRC\/} (2006), 213--248.

\bibitem{kress1998numerical}
{\sc Kress, R.}
\newblock {\em Numerical Analysis}.
\newblock Graduate Texts in Mathematics vol 181. Springer, New York, 1998.

\bibitem{Kress1993}
{\sc Kress, R., and Sloan, I.~H.}
\newblock On the numerical solution of a logarithmic integral equation of the
  first kind for the {H}elmholtz equation.
\newblock {\em Numerische Mathematik 66}, 1 (1993), 199--214.

\bibitem{lanczos1988applied}
{\sc Lanczos, C.}
\newblock {\em Applied Analysis}.
\newblock Dover Books on Mathematics. Dover Publications, 1988.

\bibitem{lazarov2007balancing}
{\sc Lazarov, R.~D., Lu, S., and Pereverzev, S.~V.}
\newblock On the balancing principle for some problems of numerical analysis.
\newblock {\em Numerische Mathematik 106}, 4 (2007), 659--689.

\bibitem{lu2010discrepancy}
{\sc Lu, S., Pereverzev, S.~V., Shao, Y., and Tautenhahn, U.}
\newblock Discrepancy curves for multi-parameter regularization.
\newblock {\em Journal of Inverse and Ill-posed Problems 18}, 6 (2010),
  655--676.

\bibitem{Parameterchoice2015}
{\sc Pereverzyev, S.~V., Sloan, I.~H., and Tkachenko, P.}
\newblock Parameter choice strategies for least-squares approximation of noisy
  smooth functions on the sphere.
\newblock {\em SIAM Journal on Numerical Analysis 53}, 2 (2015), 820--835.

\bibitem{powell1981approximation}
{\sc Powell, M. J.~D.}
\newblock {\em Approximation Theory and Methods}.
\newblock Cambridge University Press, 1981.

\bibitem{pr?ssdorf1991numerical}
{\sc Pr{\"o}ssdorf, S., and Silbermann, B.}
\newblock {\em Numerical Analysis for Integral and Related Operator Equations}.
\newblock Operator theory. Springer, 1991.

\bibitem{SLOAN1995238}
{\sc Sloan, I.}
\newblock Polynomial interpolation and hyperinterpolation over general regions.
\newblock {\em Journal of Approximation Theory 83}, 2 (1995), 238--254.

\bibitem{stein2003fourier}
{\sc Stein, E., and Shakarchi, R.}
\newblock {\em Fourier Analysis: An Introduction}.
\newblock Princeton Lectures in Analysis. Princeton University Press, 2003.

\bibitem{Stilson1996AliasFreeDS}
{\sc Stilson, T.~S., and Smith, J.~O.}
\newblock Alias-free digital synthesis of classic analog waveforms.
\newblock In {\em International Conference on Mathematics and Computing\/}
  (1996).

\bibitem{Lasso1996}
{\sc Tibshirani, R.}
\newblock Regression shrinkage and selection via the {L}asso.
\newblock {\em Journal of the Royal Statistical Society: Series B
  (Methodological) 58}, 1 (1996), 267--288.

\bibitem{Tibshirani2011}
{\sc Tibshirani, R.}
\newblock Regression shrinkage and selection via the {L}asso: a retrospective.
\newblock {\em Journal of the Royal Statistical Society: Series B (Statistical
  Methodology) 73}, 3 (2011), 273--282.

\bibitem{trefethen2013approximation}
{\sc Trefethen, L.~N.}
\newblock {\em Approximation Theory and Approximation Practice}, vol.~128.
\newblock SIAM, Philadelphia, 2013.

\bibitem{ExpTrapezoidal}
{\sc Trefethen, L.~N., and Weideman, J. A.~C.}
\newblock The exponentially convergent trapezoidal rule.
\newblock {\em SIAM Review 56}, 3 (2014), 385--458.

\bibitem{ExtPeriodic2015}
{\sc Wright, G.~B., Javed, M., Montanelli, H., and Trefethen, L.~N.}
\newblock Extension of {C}hebfun to periodic functions.
\newblock {\em SIAM Journal on Scientific Computing 37}, 5 (2015), C554--C573.

\bibitem{zygmund2002trigonometric}
{\sc Zygmund, A.}
\newblock {\em Trigonometric Series}.
\newblock Cambridge Mathematical Library. Cambridge University Press, 2002.

\end{thebibliography}

\end{document}